\documentclass[final]{siamltexmm}

\title{Frequency locking by external forcing in systems with rotational
symmetry}

\author{Lutz Recke\footnotemark[2]\ \footnotemark[4]
\and     Anatoly Samoilenko\footnotemark[3]\ \footnotemark[5]
\and     Viktor Tkachenko\footnotemark[3]\ \footnotemark[5]
\and     Serhiy Yanchuk\footnotemark[2]\ \footnotemark[6]}

\usepackage{amsmath}
\usepackage{amssymb}
\usepackage{hyperref}

\newtheorem{remark}{Remark}

\makeatletter
\begin{document}

\maketitle
\newcommand{\slugmaster}{%
\slugger{MMedia}{xxxx}{xx}{x}{x--x}}

\renewcommand{\thefootnote}{\fnsymbol{footnote}}

\footnotetext[2]{Institute of Mathematics, Humboldt University of Berlin, Unter
den Linden 6, 10099 Berlin, Germany}
\footnotetext[3]{Institute of Mathematics, National Academy of Sciences of
Ukraine, Tereschenkivska St. 3, 01601 Kiev, Ukraine}
\footnotetext[4]{This work was partially supported by the DFG Research Centre MATHEON ''Mathematics for key technologies'' under the project D8}
\footnotetext[5]{This work was partially supported by the DFG cooperation project between Germany and Ukraine WO 891/3-1}
\footnotetext[6]{This work was partially supported by the DFG Research Centre MATHEON ''Mathematics for key technologies'' under the project D21.}

\renewcommand{\thefootnote}{\arabic{footnote}}



%

\begin{abstract}
We study locking of the modulation frequency of a relative periodic orbit
 in a general $S^{1}$-equivariant system of ordinary differential equations
under an external forcing of  modulated wave type.
Our main result describes the shape of the locking region in the three-dimensional
space of the forcing parameters: intensity, wave frequency, and modulation
frequency. The difference of the wave frequencies of the relative periodic
orbit and the forcing is assumed to be large
and differences of modulation frequencies to be small.
The intensity of the forcing is small in the generic case and can be large
in the degenerate case, when the first order averaging vanishes.
Applications are external electrical and/or optical forcing of
selfpulsating states of lasers.
\end{abstract}



\pagestyle{myheadings}
\thispagestyle{plain}
\markboth{L. RECKE, A. SAMOILENKO, V. TKACHENKO, S. YANCHUK}{FREQUENCY LOCKING BY EXTERNAL FORCING
 IN SYSTEMS WITH ROTATIONAL
SYMMETRY}

\section{Introduction}
This paper concerns systems of ordinary differential equations
of the type
\begin{equation}
\frac{dx}{dt}=f(x)+\gamma g(x,\beta t,\alpha t).\label{01}
\end{equation}
Here $f:\mathbb{R}^{n}\to\mathbb{R}^{n}$ and $g:\mathbb{R}^{n}\times\mathbb{R}^{2}\to\mathbb{R}^{n}$
are $C^{l}$-smooth with $l>3$, and $\alpha>0,$ $\beta>0,$ $\gamma\ge0$
are parameters. The vector field $f$ is supposed to be $S^{1}$-equivariant
\begin{equation}
f(e^{A\xi}x)=e^{A\xi}f(x)\label{eq:symf}
\end{equation}
for all  $x\in\mathbb{R}^{n}$ and $\xi\in\mathbb{R}$,
where $A$ is a non-zero real $n\times n$-matrix such that $A^{T}=-A$
and $e^{2\pi A}=I$. The forcing term $g$ is supposed to be $2\pi$-periodic
with respect to the second and third arguments
$$
g(x,\psi+2\pi,\varphi)=g(x,\psi,\varphi)=g(x,\psi,\varphi+2\pi),
$$
and to possess the following symmetry
\begin{equation}
g(e^{A\xi}x,\psi,\varphi+\xi)=e^{A\xi}g(x,\psi,\varphi) \label{eq:symg}
\end{equation}
for all $x\in\mathbb{R}^{n}$ and $\varphi,\psi,\xi\in\mathbb{R}.$
Finally we assume that the unperturbed system
\begin{equation}
\frac{dx}{dt}=f(x) \label{02}
\end{equation}
has an exponentially orbitally stable quasiperiodic solution of the type
\begin{equation}
x(t)=e^{A\alpha_{0}t} x_{0}(\beta_{0}t).
\label{qp}
\end{equation}
Here $x_{0}: \mathbb{R}\to\mathbb{R}^{n}$ is $2\pi$-periodic, and $\alpha_{0}>0$,
 $\beta_{0}>0$ are the wave and modulation frequencies of the
solution (\ref{qp}).
The main property of quasiperiodic solutions of the type (\ref{qp}) is that
they are motions on a 2-torus without frequency locking. Those solutions
are sometimes called modulated waves \cite{Rand1982} or modulated rotating waves
\cite{Crawford1988} or relative periodic orbits \cite{Field2007}.
We assume that the following nondegeneracy condition holds:
\begin{equation}
\mbox{The vectors}\ Ax_{0}(\psi)\ \mbox{and} \ \frac{dx_{0}}{d\psi}(\psi)
\ \mbox{are linearly independent.}\label{eq:nondeg}
\end{equation}
It is easy to verify that (\ref{eq:nondeg}) is true for all $\psi\in\mathbb{R}$
if it holds for one $\psi$. Condition (\ref{eq:nondeg}) implies
that the set
\begin{equation}
\mathcal{T}_{2}:=\{(e^{A\varphi}x_{0}(\psi))\in\mathbb{R}^{n}:\quad\varphi,\psi\in\mathbb{R}\},\label{eq:torus}
\end{equation}
 which is invariant with respect to the flow corresponding to (\ref{02}),
is a two-dimensional torus.

Our main results describe open sets in the three-dimensional space
of the control parameters $\alpha$, $\beta$ and $\gamma$ with $|\alpha-\alpha_{0}|\gg1$
and $\beta\approx\beta_{0}$, where stable locking of the modulation
frequencies $\beta$ of the forcing and $\beta_{0}$ of the
 solution (\ref{qp}) occurs, i.e. where the following holds:
For almost any solution $x(t)$ to (\ref{01}), which is at a certain
moment close to $\mathcal{T}_{2}$, there exists $\sigma\in\mathbb{R}$
such that
\begin{equation}
\inf_{\psi}\|x(t)-e^{A\psi}x_{0}(\beta t+\sigma)\|\approx0\mbox{ for large }t.\label{eq:cond}
\end{equation}

Our results essentially differ in the so-called non-degenerate and
degenerate cases (see Section 2). The degenerate case includes all
cases when the averaged (with respect to the fast wave oscillations)
forcing
\begin{equation}
g_{1}(x,\psi):=\frac{1}{2\pi}\int_{0}^{2\pi}g(x,\psi,\varphi)d\varphi\label{av}
\end{equation}
vanishes identically, for example, if $g$ is of the type
$g(x,\psi,\varphi)=e^{A\varphi}\bar{g}(\psi)$.
Roughly speaking, under additional generic conditions the following
is true:

{\it Non-degenerate case:} If $\alpha$ is sufficiently large and
$\gamma$ is sufficiently small and $\beta-\beta_{0}$ is of order
$\gamma$, then locking occurs.

{\it Degenerate case:} If $\alpha$ is sufficiently large and $\gamma^{2}/\alpha$
is sufficiently small and $\beta-\beta_{0}$ is of order $\gamma^{2}/\alpha$,
then locking occurs.


The present paper extends previous work on this topic for particular
types of the vector field $f$ and the forces $g$ \cite{Schneider2005,Samoilenko2005,Recke2011}.
In particular, the considered special cases in \cite{Schneider2005,Samoilenko2005,Recke2011}
are doubly degenerate in the sense that not only the averaged forcing
(\ref{av}) vanishes but also the second averaging term turns to zero.
Remark that in \cite{Recke1998} related results are described for
the case that both differences of modulation and wave frequencies
are small, and \cite{Recke1998a} concerns the case when the internal
state as well as the external forcing are not modulated. For an even
more abstract setting of these results see \cite{Chillingworth2000}.

Systems of the type (\ref{01}) appear as models for the dynamical
behavior of self-pulsating lasers under the influence of external
periodically modulated optical and/or electrical signals, see, e.g.
\cite{Radziunas2006,Bandelow1998,Lichtner2007,Nizette2001,Peterhof1999,Sieber2002,Wieczorek2005},
and for related experimental results see \cite{Feiste1994,Sartorius1998}.
In (\ref{01}) the state variable $x$ describes the electron density
and the optical field of the laser. In particular, the Euclidian norm
$\|x\|$ describes the sum of the electron density and the intensity
of the optical field. The $S^{1}$-equivariance of (\ref{02}) is
the result of the invariance of autonomous optical models with respect
to shifts of optical phases. The solution (\ref{qp}) describes a
so-called self-pulsating state of the laser in the case that the laser
is driven by electric currents which are constant in time. In those
states the electron density and the intensity of the optical field
are time periodic with the same frequency. Self-pulsating states usually
appear as a result of Hopf bifurcations from so-called continuous
wave states, where the electron density and the intensity of the optical
field are constant in time.

External forces of the type $\gamma e^{i\alpha t}\bar{g}(\beta t)$
appear for describing an external optical injection with optical frequency
$\alpha$ and modulation frequency $\beta$. In spatially extended laser models those forces
appear as inhomogeneities in the boundary conditions. After homogenization
of those boundary conditions and finite dimensional mode approximations
(or Galerkin schemes) one ends up with systems of type (\ref{01})
with general forces of the type $\gamma g(x,\beta t,\alpha t)$
with (\ref{eq:symg}). External forces of the type
\[
g(x,\beta t,\alpha t)=g_1(x,\beta t)\mbox{ with }g_1(e^{A\xi}x,\psi)=e^{A\xi}
g_1(x,\psi)
\]
appear for describing an external electrical injection with modulation
frequency $\beta$.\\

The further structure of our paper is as follows. In Section \ref{sec:Main-results},
we present and discuss the main results. In Section \ref{sec:Averaging},
the averaging transformation is used to simplify system (\ref{01}).
The necessary properties of the unperturbed system (\ref{02}) are
discussed in Section \ref{sec:Unperturbed-system}. In Section \ref{sec:Local-coordinates},
we introduce local coordinates in the vicinity of the unperturbed
invariant torus. Further, we show the existence of the perturbed
integral manifold and study the dynamics on this manifold in Section \ref{sec:Investigation-of-the}.
Remaining proofs of main results are given in Section \ref{sec:Proofs-o-theorems}.
Finally, in Sections \ref{sec:example1} and \ref{sec:example2}
we illustrate our theory by two examples.

\section{\label{sec:Main-results}Main results}

In the co-rotating coordinates $x(t)=e^{A\alpha_{0}t}y(\beta_{0}t)$
the unperturbed equation (\ref{02}) reads as
\begin{equation}
\beta_{0}\frac{dy}{d\psi}=f(y)-\alpha_{0}Ay.\label{newcoord}
\end{equation}
 The quasiperiodic solution (\ref{qp}) to (\ref{02}) now appears
as $2\pi$-periodic solution $y(\psi)=x_{0}(\psi)$ to (\ref{newcoord}).
The corresponding variational equation around this solution is
\begin{equation}
\beta_{0}\frac{dy}{d\psi}=\left(f'(x_{0}(\psi))-\alpha_{0}A\right)y.\label{0per1}
\end{equation}
It is easy to verify (see Section \ref{sec:Unperturbed-system}),
that (\ref{0per1}) has two linear independent (because of assumption
(\ref{eq:nondeg})) periodic solutions
\begin{equation}
q_{1}(\psi):=\frac{dx_{0}}{d\psi}(\psi),\quad
q_{2}(\psi):=Ax_{0}(\psi),\label{q}
\end{equation}
 which correspond to the two trivial Floquet multipliers 1 of the
periodic solution $x_{0}$ to (\ref{newcoord}). One of these Floquet
multipliers appears because of the $S^{1}$-equivariance of (\ref{newcoord}),
and the other one because (\ref{newcoord}) is autonomous. From the
exponential orbital stability of (\ref{qp}) it follows that the trivial
Floquet multiplier 1 of the periodic solution $x_{0}$ to (\ref{newcoord})
has multiplicity two, and the absolute values of all other multipliers
are less than one. Therefore, there exist two solutions $p_{1}(\psi)$
and $p_{2}(\psi)$ to the adjoint variational equation
\begin{equation}
\beta_{0}\frac{dz}{d\psi}=-\left(f'(x_{0}(\psi))^{T}+\alpha_{0}A\right)z\label{0per2}
\end{equation}
 with
\begin{equation}
\label{24new}
p_{j}^{T}(\psi)q_{k}(\psi)=\delta_{jk}
\end{equation}
 for all $j,k=1,2$ and $\psi\in\mathbb{R}$, where $\delta_{jk}=1$
for $j=k$ and $\delta_{jk}=0$ otherwise.

\subsection{\emph{Non-degenerate case}}

Using notation (\ref{av}) we define a $2\pi$-periodic function $G_{1}:\mathbb{R}\to\mathbb{R}$
by
\begin{equation}
G_{1}(\psi):=\frac{1}{2\pi}\int\limits _{0}^{2\pi}p_{1}^{T}(\psi+\theta)g_{1}(x_{0}(\psi+\theta),\theta)d\theta\label{0per3}
\end{equation}
 and its maximum and minimum as
\[
G_{1}^{+}:=\max_{\psi\in[0,2\pi]}G_{1}(\psi),\quad G_{1}^{-}:=\min_{\psi\in[0,2\pi]}G_{1}(\psi).
\]
For the sake of simplicity we will suppose that all singular points
of $G_{1}$ are non-degenerate, i.e.
\[
G_{1}''(\psi)\not=0\mbox{ for all }\psi\mbox{ with }G_{1}'(\psi)=0.
\]
This implies that the set of singular points of $G_{1}$ consists
of an even number $2N$ of different points. The set of singular values
of $G_{1}$ will be denoted by
\[
S_{1}:=\{G_{1}(\psi_{1}),\ldots,G_{1}(\psi_{2N})\},\mbox{ where }\{\psi_{1},\ldots,\psi_{2N}\}=\{\psi\in[0,2\pi):\, G'_{1}(\psi)=0\}.
\]

Our first result describes the behavior (under the perturbation of
the forcing term in (\ref{01})) of $\mathcal{T}_{2}\times\mathbb{R}$,
which is an integral manifold to (\ref{02}), as well as the dynamics
of (\ref{01}) on the perturbed integral manifold to the leading order.

\begin{theorem}
\label{theorem01} For all $\beta_{1}<\beta_{2}$ there exist positive
$\gamma_{*}$ and $\alpha_{*}$ such that for all $\alpha, \beta$, and $\gamma$
satisfying
\begin{equation}
\alpha\ge\alpha_{*}, \quad \beta_1 \le \beta \le \beta_2, \quad 0 \le \gamma\le\gamma_{*},\quad\label{cond1}
\end{equation}
system (\ref{01}) has a three-dimensional integral manifold $\mathfrak{M}_{1}(\alpha,\beta,\gamma)$,
which can be parameterized by $\psi,\varphi,t\in\mathbb{R}$ in the
form
\begin{equation}
x=e^{A\varphi}x_{0}(\psi)+\gamma e^{A\varphi} U_{1}\left(\psi,\beta t,\gamma\right)+\frac{\gamma}{\alpha}U_{2}\left(\psi,\varphi,\beta t,\alpha t,\gamma,
\alpha^{-1}
\right),
\label{eq:IM0}
\end{equation}
where functions
$U_{1}$ and $U_2$ are $C^{l-2}$ smooth and $2\pi$-periodic in
$\psi,\varphi,\beta t$, and $\alpha t.$

The manifold $\mathfrak{M}_{1}(\alpha,\beta,\gamma)$ is orbitally
asymptotically stable and solutions on this manifold are determined
by the following system
\begin{eqnarray}
\frac{d\psi}{dt} & = & \beta_{0}+\gamma p_{1}^{T}(\psi)g_{1}(x_{0}(\psi),\beta t)+\gamma^{2}S_{11}+\frac{\gamma}{\alpha}S_{12},\label{tau1}\\
\frac{d\varphi}{dt} & = & \alpha_{0}+\gamma p_{2}^{T}(\psi)g_{1}(x_{0}(\psi),\beta t)+\gamma^{2}S_{21}+\frac{\gamma}{\alpha}S_{22},\label{tau2}
\end{eqnarray}
where functions
$S_{jk} = S_{jk}(\psi,\varphi,\beta t,\alpha t,\gamma,\alpha^{-1})$
 are $C^{l-2}$ smooth and $2\pi$-periodic in $\psi,\varphi,\beta t$, and $\alpha t.$

In addition, for any $\epsilon>0$ there exist positive $\alpha_1,  \gamma_1$, and $\delta$ such that for all
$\alpha, $ $\beta$, and $\gamma$ satisfying
\begin{equation}
\alpha \ge \alpha_1, \ 0 \le \gamma \le \gamma_1, \ \gamma G_{1}^{-}<\beta-\beta_{0}<\gamma G_{1}^{+},
\ \mathrm{dist}\left(\frac{\beta-\beta_{0}}{\gamma},\ S_{1}\right)\ge\epsilon
\label{cond2}
\end{equation}
the following statements hold:

1. System (\ref{01}) has
integral manifold $\mathfrak{M}_{1}(\alpha,\beta,\gamma)$ of the form
(\ref{eq:IM0}) and
 even number of two-dimensional integral submanifolds
$\mathfrak{N}_{j}\subset\mathfrak{M}_{1}(\alpha,\beta,\gamma),$ $j=1,...,2\tilde{N},$
$\tilde{N}=\tilde{N}(\alpha,\beta,\gamma)\le N,$ which are parameterized
by $\varphi,t\in\mathbb{R},$ in the form
\begin{eqnarray}
x & = & e^{A\varphi}x_{0}(\beta t+\vartheta_{j})+\gamma V_{1j}(\varphi,\beta t, \alpha t, \gamma, \alpha^{-1})+\frac{1}{\alpha}V_{2j}(\varphi,\beta t,\alpha t,\gamma,\alpha^{-1}).\label{xm}
\end{eqnarray}
The dynamic on the manifold $\mathfrak{N}_{j}$ is determined by a system of the type
\[
\frac{d\varphi}{dt}=\alpha_{0} + \gamma W_{j}\left(\varphi,\beta t,\alpha t,\gamma,\alpha^{-1}\right).
\]
Here $\vartheta_j$ are constants and functions
$V_{1j}$, $V_{2j}$, and $W_{j}$
are $C^{l-2}$ smooth and $2\pi$-periodic in
$\psi,\varphi,\beta t$, and $\alpha t.$

2. Every solution $x(t)$ of system (\ref{01}) that at a certain
moment of time $t_{0}$ belongs to a $\delta$-neighborhood of the
torus $\mathcal{T}_{2}$ tends to one of the manifolds $\mathfrak{N}_{j}$
as $t\to+\infty.$

\end{theorem}

One of the main statements of this theorem is that under conditions
 (\ref{cond2}), within the perturbed
manifold $\mathfrak{M}_{1}(\alpha,\beta,\gamma)$, there exist lower-dimensional
{}``resonant'' manifolds $\mathfrak{N}_{j}$ corresponding to the
frequency locking. These manifolds attract all solutions from the
neighborhood of $\mathcal{T}_{2}$. Hence, the asymptotic behavior
of solutions is described by (\ref{xm}) and has the modulation frequency
$\beta$. One can observe also from (\ref{xm}), that the perturbed
dynamics is, in the leading order, a motion along $\mathcal{T}_{2}$
with the new modulation frequency. The following theorem describes
these locking properties more precisely.

\begin{theorem}
\label{thm:main} For any $\epsilon>0$ and $\epsilon_{1}>0$
there exist positive $\gamma_1, \alpha_1$, and $\delta$ such
that for all parameters $(\alpha,\beta,\gamma)$ satisfying the conditions
 (\ref{cond2}) and for any solution
$x(t)$ of system (\ref{01}) such that $\mbox{{\rm dist}}(x(t_{0}),\mathcal{T}_{2})<\delta$
for certain $t_{0}\in\mathbb{R}$ there exist $\sigma,T\in\mathbb{R}$
such that
\begin{equation}
\inf_{\phi}\|x(t)-e^{A\phi}x_{0}(\beta t+\sigma)\|<\epsilon_{1}\mbox{ for all }t>T.\label{theo3}
\end{equation}
\end{theorem}

The conditions for the locking  (\ref{cond2})
depend on the properties of the function $G_{1}$.
For the case that $S_{1}$ consists of two numbers $G_{1}^{-}$ and
$G_{1}^{+}$ only, i.e. that $G_{1}$ has only two singular values,
the set of parameters satisfying these conditions is illustrated in
Fig. \ref{fig:generic} (left) for a fixed value of $\alpha\ge\alpha_1$.
The admissible values of the parameters $\beta$ and $\gamma$
belong to a cone with linear boundaries
\[
\beta=\beta_{0}+\gamma\left(G_{1}^{-}+\epsilon\right)\mbox{ and }\beta=\beta_{0}+\gamma\left(G_{1}^{+}-\epsilon\right),
\]
 bounded from above by $\gamma\le\gamma_1$.

If $G_{1}$ has more than two singular values, then the corresponding
regions in the parameter space
\[
\left|\beta-\beta_{0}-\gamma G_{1}(\psi_{j})\right|<\gamma\epsilon
\]
should be excluded. Here $\psi_{j}$ are the additional singular points
of $G_{1}$. An example is shown in Fig.~\ref{fig:generic} (right)
for the case with two additional singular points.
\begin{figure}
\begin{centering}
\includegraphics[width=0.8\linewidth]{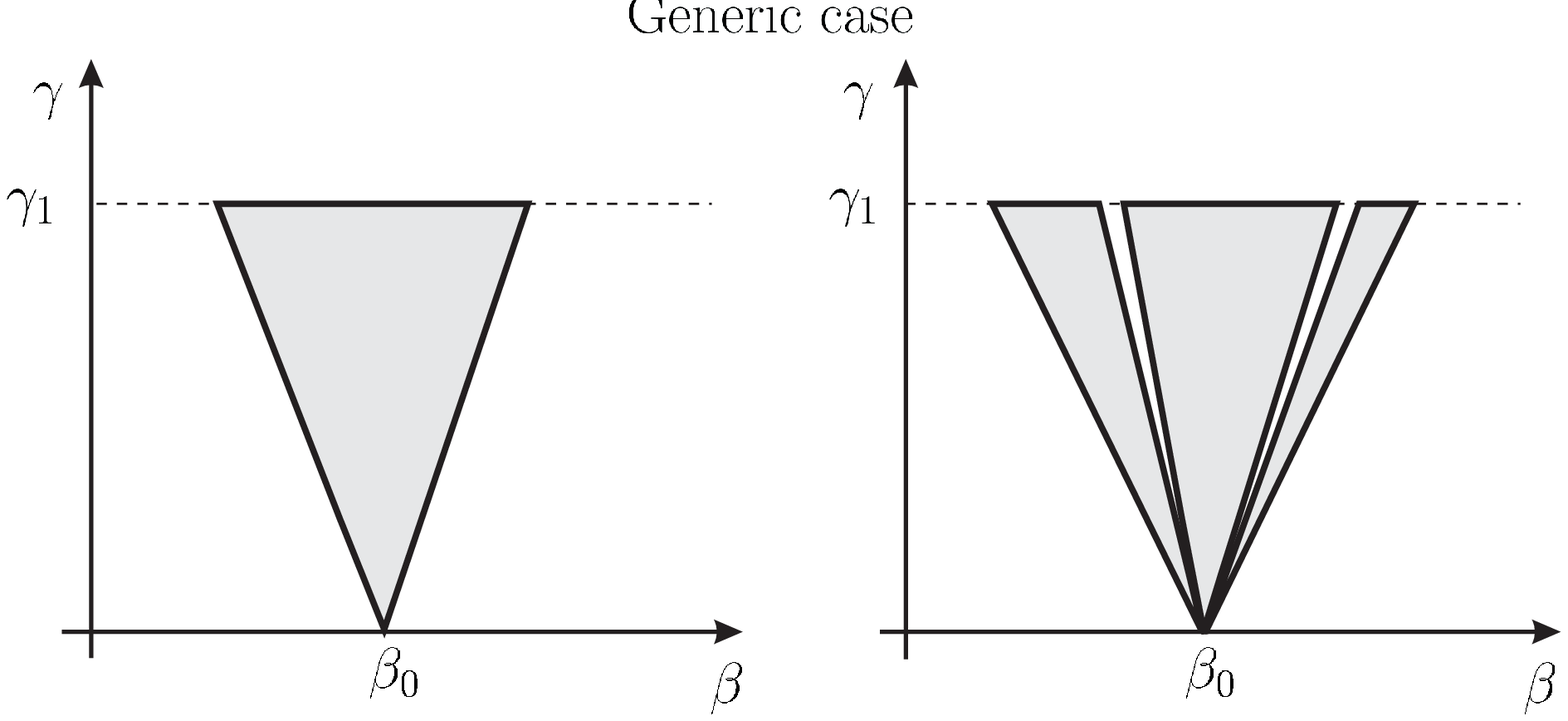}
\par\end{centering}

\caption{\label{fig:generic}Cross-section of the locking regions with respect to the frequency
$\beta$ and amplitude $\gamma$ by $\alpha={\mathrm{const}}$
of the external perturbation for
the generic case $g_{1}\not\equiv0$. }
\end{figure}

\subsection{Degenerate case}

In this subsection we suppose that
\begin{equation}
\frac{1}{2\pi}\int_{0}^{2\pi}g(x,\psi,\varphi)d\varphi=0\mbox{ for all }x\in\mathbb{R}^{n}\mbox{ and }\psi\in\mathbb{R}.\label{deg}
\end{equation}
In this case the functions $g_{1}$ and $G_{1}$, which are defined in (\ref{av})
and (\ref{0per3}), are identically zero and, hence, cannot give any
information about locking behavior like in Theorem \ref{theorem01}.
Instead, the following functions $g_{2}:\mathbb{R}^{n}\times\mathbb{R}\to\mathbb{R}^{n}$
and $G_{2}:\mathbb{R}\to\mathbb{R}$ will define the locking:
\[
g_{2}(x,\psi):=-
\frac{1}{2\pi}\int_{0}^{2\pi}
\frac{\partial u_1}{\partial x}(x,\psi,\varphi)
g(x,\psi,\varphi) d\varphi
\]
 and
\begin{equation}
G_{2}(\psi):=\frac{1}{2\pi}\int\limits _{0}^{2\pi}p_{1}^{T}(\psi+\theta)g_{2}(x_{0}(\psi+\theta),\theta)d\theta.\label{01per3}
\end{equation}
Here
$$
u_1(x,\psi,\varphi) = \int^{\varphi}
g(x,\psi,\theta)d\theta
$$
is the antiderivative of $g(x,\psi,\varphi)$ satisfying
\[
\int_{0}^{2\pi}u_1(x,\psi,\varphi)d\varphi=0.
\]
 Now we proceed as in the non-degenerate case: We denote
\[
G_{2}^{+}:=\max_{\psi\in[0,2\pi]}G_{2}(\psi),\quad G_{2}^{-}:=\min_{\psi\in[0,2\pi]}G_{2}(\psi)
\]
 and suppose that all singular points of $G_{2}$ are non-degenerate
\[
G_{2}''(\psi)\not=0\mbox{ for all }\psi\mbox{ with }G_{2}'(\psi)=0.
\]
This implies that the set of singular points of $G_{2}$ consists
of an even number $2N$ of different points. The set of singular values
of $G_{2}$ will be denoted by
\[
S_{2}:=\{G_{2}(\psi_{1}),\ldots,G_{2}(\psi_{2N})\},\mbox{ where }\{\psi_{1},\ldots,\psi_{2N}\}=\{\psi\in[0,2\pi):\, G'_{2}(\psi)=0\}.
\]

Our next result describes the locking behavior of the dynamics close
to ${\mathcal{T}}_{2}$ in the degenerate case.

\begin{theorem}
\label{theorem011} Suppose (\ref{deg}) holds.
Then for all $\beta_1 < \beta_2$
there exist positive
constants $\mu_{*}$ and $\alpha_{*}$ such that for all $\alpha, \beta$,
and $\gamma$ satisfying
\begin{equation}
\beta_1 \le \beta \le \beta_2, \ \gamma^{2}/\alpha\le\mu_{*}, \ \alpha\ge\alpha_{*},\label{eq:cond00}
\end{equation}
system (\ref{01}) has a three-dimensional integral manifold $\mathfrak{M}_{2}(\alpha,\beta,\gamma)$,
which can be parameterized by $\psi,\varphi,t\in\mathbb{R}$ in the
form
\begin{equation}
x=e^{A\varphi}x_{0}(\psi)+\frac{\gamma}{\alpha}\tilde{U}_{1}\left(\psi,\varphi,\beta t,\alpha t,
\frac{\gamma^{2}}{\alpha}
,
\frac{1}{\alpha}
\right)
+\frac{\gamma^{2}}{\alpha}\tilde{U}_{2}\left(\psi,\varphi,\beta t,\alpha t,\frac{\gamma^{2}}{\alpha},\frac{1}{\alpha}\right),\label{eq:IM0-deg}
\end{equation}
where functions $\tilde{U}_{1}$ and $\tilde{U}_{2}$
 are $C^{l-3}$ smooth
and $2\pi$-periodic in $\psi,\varphi,\beta t$, and $\alpha t.$

The manifold $\mathfrak{M}_{2}(\alpha,\beta,\gamma)$ is orbitally
asymptotically stable and dynamics on this manifold is determined
by the following system
\begin{eqnarray}
\frac{d\psi}{dt} & = & \beta_{0}+\frac{\gamma^{2}}{\alpha}p_{1}^{T}(\psi){g}_{2}(x_{0}(\psi),\beta t)+\frac{\gamma}{\alpha^{2}}\tilde{S}_{11}+\frac{\gamma^{2}}{\alpha^{2}}\tilde{S}_{12}+\frac{\gamma^{4}}{\alpha^{2}}\tilde{S}_{13},\label{eq:psi}\\
\frac{d\varphi}{dt} & = & \alpha_{0}+\frac{\gamma^{2}}{\alpha}p_{2}^{T}(\psi){g}_{2}(x_{0}(\psi),\beta t)+\frac{\gamma}{\alpha^{2}}\tilde{S}_{21}+\frac{\gamma^{2}}{\alpha^{2}}\tilde{S}_{22}+\frac{\gamma^{4}}{\alpha^{2}}\tilde{S}_{23},\label{eq:theta}
\end{eqnarray}
where functions $\tilde{S}_{mk} = (\psi,\varphi,\beta t,\gamma^{2}/\alpha,1/\alpha)$
 are $C^{l-3}$ smooth
and $2\pi$-periodic in $\psi,\varphi,\beta t$, and $\alpha t.$

In addition, for any $\epsilon>0$, there exists some positive
$\alpha_1, c_{1}$, $c_2$, and $\delta$ such that for all $\alpha, \beta$, and $\gamma,$
satisfying
\begin{equation}
\alpha \ge \alpha_1, \  \frac{c_1}{\alpha}\le\gamma\le c_2\sqrt{\alpha}, \
\frac{\gamma^{2}}{\alpha}G_{2}^{-}<\beta-\beta_{0}<\frac{\gamma^{2}}{\alpha}G_{2}^{+}, \
\mathrm{dist}\left(\frac{\alpha}{\gamma^{2}}(\beta-\beta_{0}),S_{2}\right)\ge\epsilon,\label{cond33}
\end{equation}
then the following statements hold:

 1. System (\ref{01}) has
integral manifold $\mathfrak{M}_{2}(\alpha,\beta,\gamma)$ of the form
(\ref{eq:IM0-deg}) and
 even number of two-dimensional integral submanifolds
$\mathfrak{L}_{j}\subset\mathfrak{M}_{2},$ $j=1,...,2\tilde{N}, \ \tilde{N} = \tilde{N}(\alpha, \beta, \gamma) \le N,$
which are parametrized by $\varphi, t \in \mathbb{R}$ in the form
\begin{equation}
x=e^{A\varphi}x_{0}(\beta t+\vartheta_{j})+\frac{\gamma}{\alpha}\tilde{V}_{1j} + \frac{\gamma^{2}}{\alpha}\tilde{V}_{2j}
+ \frac{1}{\alpha}\tilde{V}_{3j} +  \frac{1}{\alpha\gamma}\tilde{V}_{4j},   \label{xm-deg}
\end{equation}
 where  $C^{l-3}$ smooth functions
 $\tilde{V}_{ij} = \tilde{V}_{1j}(\varphi,\beta t,\alpha t,{\gamma^{2}}/{\alpha},{1}/{\alpha}, {1}/{\alpha\gamma})$
 are $2\pi$-periodic in $\varphi, \beta t$, and $\alpha t.$
 The system on the manifold $\mathfrak{L}_{j}$ reduces to
\[
\frac{d\varphi}{dt}=\alpha_{0}+\frac{\gamma^{2}}{\alpha}\tilde{W}_{j1}+\frac{\gamma}{\alpha^{2}}\tilde{W}_{j2}+\frac{1}{\alpha^{3}}\tilde{W}_{j3},
\]
where functions $\tilde{W}_{j1} = \tilde{W}_{j1}(\varphi,\beta t,\alpha t,{\gamma^{2}}/{\alpha},{1}/{\alpha}, {1/\alpha\gamma})$
are $C^{l-3}$ smooth
and $2\pi$-periodic in $\varphi,\beta t$, and $\alpha t.$

 2. Every solution $x(t)$ of system (\ref{01}) that at a certain moment of
time $t_{0}$ belongs to $\delta$-neighborhood of the torus $\mathcal{T}_{2}$
tends to one of the manifolds $\mathfrak{L}_{j}$ as $t\to+\infty.$
\end{theorem}

Theorem \ref{theorem011} gives verifiable conditions on the parameters
$(\alpha,\beta,\gamma)$, for which the locking of modulation frequency
to the modulation frequency $\beta$ of the perturbation takes place
for the case when the degeneracy condition $g_{1}\equiv0$ is fulfilled.
These conditions are given by  (\ref{cond33})
and differ from the conditions of locking in the generic case given
by Theorem \ref{theorem01}. The locking phenomenon in the leading
order looks similarly in both cases, i.e. the solutions tend asymptotically
to $\mathcal{T}_{2}$ and modulation frequency $\beta$. More precisely,
the following theorem holds.

\begin{theorem}
\label{thm:main2-1} For any $\epsilon>0$ and $\epsilon_{1}>0$
there exist positive  $\alpha_{1}$,  $c_{1}$, $c_2$, and $\delta$
such that for all parameters $(\alpha,\beta,\gamma)$ satisfying the
conditions  (\ref{cond33}) and
for any solution $x(t)$ of system (\ref{01}) such that $\mbox{{\rm dist}}(x(t_{0}),\mathcal{T}_{2})<\delta$
for certain $t_{0}\in\mathbb{R}$ there exist $\sigma,T\in\mathbb{R}$
such that
\begin{equation}
\inf_{\phi}\|x(t)-e^{A\psi}x_{0}(\beta t+\sigma)\|<\epsilon_{1}\mbox{ for all }t>T.\label{theo6}
\end{equation}
\end{theorem}

Let us illustrate the set of parameters  (\ref{cond33})
leading to the locking and compare it with the generic case. For a
fixed large enough $\alpha$, the region in the parameter space $(\beta,\gamma)$
has again a shape of a cone as in Fig.~\ref{fig:generic}, but now
the cone has tangent boundaries
\begin{equation}
\beta=\beta_{0}+\frac{\gamma^{2}}{\alpha}\left(G_{2}^{\mp}\pm\epsilon\right)\label{eq:bouny}
\end{equation}
leading to a smaller synchronization domain for moderate values of
$\gamma$, see Fig.~\ref{fig:deg}. Small values of $\gamma$ are
even excluded $\gamma>c_1/\alpha$. But, on the other hand, the
synchronization is now allowed for large values of $\gamma$ up to
$c_{2}\sqrt{\alpha}$. This means practically, that the locking
occurs not only for small but also for large amplitude perturbations,
contrary to the case $g_{1}\not\equiv0$.

Again, when the singular set $S_{2}$ of $G_{2}$ contains more than
2 points, then the corresponding regions in the parameter space
\[
\left|\beta-\beta_{0}-\frac{\gamma^{2}}{\alpha}\psi_{j}\right|<\frac{\gamma^{2}}{\alpha}\epsilon
\]
should be excluded. Here $\psi_{j}$ are additional points from $S_{2}$.
Example in Fig.~\ref{fig:deg} (right), shows the case with two additional
singular points.

\begin{figure}
\begin{centering}
\includegraphics[width=0.8\linewidth]{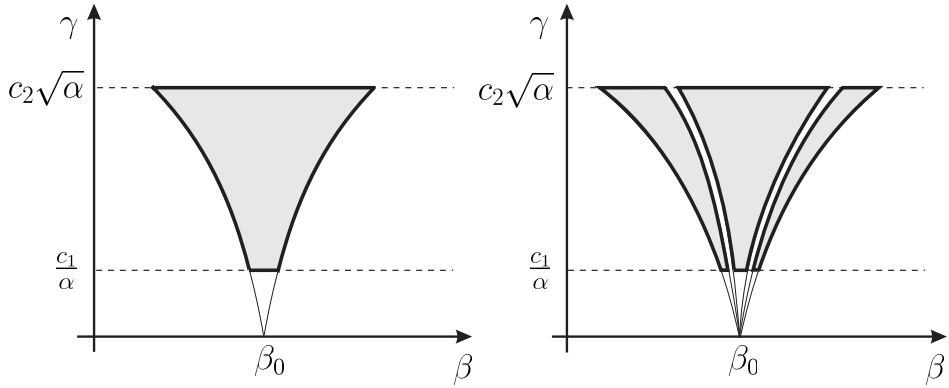}
\par\end{centering}

\caption{\label{fig:deg}Cross-sections of the locking regions with respect to the frequency $\beta$
and amplitude $\gamma$ by $\alpha=\mathrm{const}$
of the external perturbation for the degenerate
case $g_{1}\equiv0$. }
\end{figure}

Finally note that the locking phenomenon described in \cite{Recke2011}
for a simple model with symmetry corresponds to the case when an additional
degeneracy takes place with $g_{2}\equiv0$. As it is shown in \cite{Recke2011},
the locking cone becomes even more high ($\gamma\le\mu_{*}\alpha$)
and thin with the slope proportional to $\gamma^{2}/\alpha^{2}$ instead
of $\gamma^{2}/\alpha$ in (\ref{eq:bouny}).

\section{Averaging\label{sec:Averaging}}

We perform changes of variables which depend on the average of the perturbation function
$g$ with respect to fast oscillation argument $\alpha t.$
As the result of these transformations we obtain an equivalent system, where the fast oscillation
terms have the order of magnitude of $\gamma / \alpha$ and $\gamma^2 / \alpha$ and smaller.
The principles and details of the averaging procedure can be found e.g. in \cite{Bogoliubov1961,Hale1980,Sanders2007}.
This transformation has the form
\begin{eqnarray}
x & = & x_{1}+\frac{\gamma}{\alpha}u_{1}(x,\beta t,\alpha t).\label{change1}
\end{eqnarray}
 Inserting (\ref{change1}) into (\ref{01}), we obtain
\[
\frac{dx_{1}}{dt}+\frac{\gamma}{\alpha}\left(\frac{\partial u_{1}}{\partial x}\frac{dx}{dt}+\frac{\partial u_{1}}{\partial(\beta t)}\beta\right)+\gamma\frac{\partial u_{1}}{\partial(\alpha t)}=f\left(x_{1}+\frac{\gamma}{\alpha}u_{1}\right)+\gamma g(x,\beta t,\alpha t).
\]
 Accordingly to the idea of the averaging method, the terms of order
$\gamma$ depending on the high frequency $\alpha t$ should vanish
due to the choice of $u_{1}$. This leads to the condition
\begin{equation}
\frac{\partial u_{1}}{\partial(\alpha t)}=g(x,\beta t,\alpha t)-g_{1}(x,\beta t),\label{sr1}
\end{equation}
 where
\[
g_{1}(x,\beta t)=\frac{1}{2\pi}\int_{0}^{2\pi}g(x,\beta t,\varphi)d\varphi.
\]
 Hence
\[
u_{1}(x,\beta t,\alpha t)=\int_{\varphi_{0}}^{\alpha t}\left(g(x,\beta t,\varphi)-g_{1}(x,\beta t)\right)d\varphi
\]
is a periodic function in $\alpha t$ and $\beta t$.
Strictly speaking, the above integral is considered componentwise, and
the initial points $\varphi_0$ can be different for
each component of vector-function $u_{1}$.
Additionally, we choose the initial points $\varphi_{0,j}$ in such a way that
\[
\int_{0}^{2\pi}u_{1}(x,\beta t,\varphi)d\varphi=0.
\]
 The resulting averaged system reads
\begin{eqnarray}
\frac{dx_{1}}{dt} & = & f(x_{1})+\gamma g_{1}(x_{1},\beta t)+\frac{\gamma}{\alpha}A_{1}(x_{1},\beta t,\alpha t)+\frac{\gamma^{2}}{\alpha}A_{2}(x_{1},\beta t,\alpha t)\nonumber \\[2mm]
 &  & +\frac{\gamma^{2}}{\alpha^{2}}X_{1}(x_{1},\beta t,\alpha t,\frac{\gamma}{\alpha})+\frac{\gamma^{3}}{\alpha^{2}}X_{2}(x_{1},\beta t,\alpha t,\frac{\gamma}{\alpha}),\label{av01}
\end{eqnarray}
 where
\begin{eqnarray*}
A_{1}(x_{1},\beta t,\alpha t) & = & \frac{\partial f}{\partial x}u_{1}-\frac{\partial u_{1}}{\partial x}f-\frac{\partial u_{1}}{\partial(\beta t)}\beta,\\[2mm]
A_{2}(x_{1},\beta t,\alpha t) & = & -\frac{\partial u_{1}}{\partial x}g+\frac{\partial g_{1}}{\partial x}u_{1}.
\end{eqnarray*}
 Functions $A_{j}(x_{1},\beta t,\alpha t)$ and $X_{j}(x_{1},\beta t,\alpha t,\frac{\gamma}{\alpha}),j=1,2,$
are $C^{l-1}$-smooth and $2\pi$-periodic in $\beta t$ and $\alpha t.$
Here we use the fact that for small $\gamma / \alpha$ the change of variables (\ref{change1})
is equivalent to
\begin{eqnarray}
x & = & x_{1}+\frac{\gamma}{\alpha}\tilde u_{1}(x_1,\beta t,\alpha t, \frac{\gamma}{\alpha})\label{change1p}
\end{eqnarray}
with smooth bounded and periodic with respect to $\beta t$ and $\alpha t$ function $\tilde u_{1}.$

Note that the averaged term $g_{1}(x,\psi)$ is equivariant with respect
to the group action $e^{A\xi}:$
\begin{eqnarray}
g_{1}(e^{A\xi}x,\psi) & = & \frac{1}{2\pi}\int_{0}^{2\pi}g(e^{A\xi}x,\psi,\varphi)d\varphi\\[2mm]
 & = & \frac{1}{2\pi}\int_{0}^{2\pi}e^{A\xi}g(x,\psi,\varphi-\xi)d\varphi=e^{A\xi}g_{1}(x,\psi).
\end{eqnarray}

If $g_{1}(x,\psi)\equiv0$ we perform the second averaging change
of variables
\begin{eqnarray}
x_{1} & = & x_{2}+\frac{\gamma}{\alpha^{2}}u_{21}(x_{1},\beta t,\alpha t)+\frac{\gamma^{2}}{\alpha^{2}}u_{22}(x_{1},\beta t,\alpha t).\label{change2}
\end{eqnarray}
 The functions $u_{21}$ and $u_{22}$ are selected from the conditions
that the terms of the orders $\gamma/\alpha$ and $\gamma^{2}/\alpha$,
which depend on high frequency $\alpha t$, vanish. This leads to
\begin{equation}
\frac{\partial u_{21}}{\partial(\alpha t)}=A_{1}(x_{1},\beta t,\alpha t),\quad\frac{\partial u_{22}}{\partial(\alpha t)}=A_{2}(x_{1},\beta t,\alpha t)-g_{2}(x_{1},\beta t),\label{eq:yyyttt}
\end{equation}
 and
\begin{equation}
g_{2}(x_{1},\beta t)=\frac{1}{2\pi}\int_{0}^{2\pi}A_{2}(x_{1},\beta t,\varphi)d\varphi=-\frac{1}{2\pi}\int_{0}^{2\pi}\frac{\partial u_{1}(x_{1},\beta t,\varphi)}{\partial x}g(x_{1},\beta t,\varphi)d\varphi.\label{2aver}
\end{equation}
 We have used in (\ref{eq:yyyttt}) the following property of $A_{1}$:
\[
\frac{1}{2\pi}\int_{0}^{2\pi}A_{1}(x_{1},\beta t,\varphi)d\varphi=0.
\]

The second averaging function $g_{2}(x,\beta t)$ is also equivariant
with respect to the symmetry group action
\[
g_{2}(e^{A\xi}x,\psi)=e^{A\xi}g_{2}(x,\psi).
\]
 This can be seen from the following calculations
\begin{eqnarray*}
 &  & g_{2}(e^{A\xi}x,\psi)=-\frac{1}{2\pi}\int_{0}^{2\pi}\left(\int_{\varphi_{0}}^{\varphi}\frac{\partial g}{\partial x}(e^{A\xi}x,\psi,\theta)d\theta\right)g(e^{A\xi}x,\psi,\varphi)d\varphi\\[2mm]
 &  & =-\frac{1}{2\pi}\int_{0}^{2\pi}\left(\int_{\varphi_{0}}^{\varphi}e^{A\xi}\frac{\partial g}{\partial x}(x,\psi,\theta-\xi)e^{-A\xi}d\theta\right)e^{A\xi}g(x,\psi,\varphi-\xi)d\varphi\\[2mm]
 &  & =-e^{A\xi}\frac{1}{2\pi}\int_{0}^{2\pi}\left(\int_{\varphi_{0}-\xi}^{\varphi-\xi}\frac{\partial g}{\partial x}(x,\psi,\theta)d\theta\right)g(x,\psi,\varphi-\xi)d\varphi\\[2mm]
 &  & =-e^{A\xi}\frac{1}{2\pi}\int_{-\xi}^{2\pi-\xi}\left(\int_{\varphi_{0}}^{\varphi}\frac{\partial g}{\partial x}(x,\psi,\theta)d\theta\right)g(x,\psi,\varphi)d\varphi\\[2mm]
 &  & -e^{A\xi}\frac{1}{2\pi}\int_{\varphi_{0}-\xi}^{\varphi_{0}}\frac{\partial g}{\partial x}(x,\psi,\theta)d\theta\int_{-\xi}^{2\pi-\xi}g(x,\psi,\varphi)d\varphi\\[2mm]
 &  & =-e^{A\xi}\frac{1}{2\pi}\int_{0}^{2\pi}\left(\int_{\varphi_{0}}^{\varphi}\frac{\partial g}{\partial x}(x,\psi,\theta)d\theta\right)g(x,\psi,\varphi)d\varphi=e^{A\xi}g_{2}(x,\psi).
\end{eqnarray*}
 After the second averaging, in the case $g_{1}(x,\beta t)\equiv0$
, the system admits the form
\begin{equation}
\frac{dx_{2}}{dt}=f(x_{2})+\frac{\gamma^{2}}{\alpha}g_{2}(x_{2},\beta t)+\frac{\gamma}{\alpha^{2}}g_{3}(x_{2},\beta t,\alpha t,\frac{\gamma^{2}}{\alpha},\frac{1}{\alpha})+\frac{\gamma^{2}}{\alpha^{2}}g_{4}(x_{2},\beta t,\alpha t,\frac{\gamma^{2}}{\alpha},\frac{1}{\alpha}),\label{av2}
\end{equation}
 where $C^{l-2}$-smooth functions $g_{3}$ and $g_{4}$ are $2\pi$-periodic
in $\alpha t$ and $\beta t$ and function $g_{2}(x_{2},\beta t)$
is $2\pi$-periodic in $\beta t$ and equivariant with respect to
action $e^{A\xi}.$ Note that by (\ref{change1}) and (\ref{change2}) the variable $x$ is expressed by $x_2$
in a form
\begin{eqnarray} \label{change2p}
x = x_2 + \frac{\gamma}{\alpha}\tilde u_{2}(x_2,\beta t,\alpha t, \frac{\gamma^2}{\alpha}, \frac{1}{\alpha})
\end{eqnarray}
with smooth bounded and periodic with respect to $\beta t$ and $\alpha t$ function $\tilde u_{2}.$

\section{Useful properties of the unperturbed system\label{sec:Unperturbed-system}}

In this section, we consider useful properties of the linearized unperturbed
system and introduce an appropriate basis (matrix $\Phi_{0}$), which
locally splits the coordinates along the invariant torus (\ref{eq:torus})
and transverse to it. Further, this basis will be used in section
\ref{sec:Local-coordinates} for the introduction of appropriate local
coordinate system.

Since the unperturbed system (\ref{02}) has quasiperiodic solution
$\tilde{x}(t)=e^{A\alpha_{0}t}x_{0}(\beta_{0}t)$ then
\begin{equation}
\frac{dx_{0}(\beta_{0}t)}{dt}+\alpha_{0}Ax_{0}(\beta_{0}t)=f(x_{0}(\beta_{0}t)).\label{riv1}
\end{equation}
 The corresponding variational system
\begin{equation}
\frac{dy}{dt}=\frac{\partial f(\tilde{x})}{\partial x}y\label{var}
\end{equation}
 has two quasiperiodic solutions
\begin{equation}
Ae^{A\alpha_{0}t}x_{0}(\beta_{0}t)\,\,\,\ \mbox{and}\,\,\, \ e^{A\alpha_{0}t}\frac{dx_{0}(\beta_{0}t)}{dt}.\label{eq:2sol}
\end{equation}
 The following properties of the Jacobian follow from the equivariance
conditions

\begin{equation}
\frac{\partial f(x)}{\partial x}=e^{-A\xi}\frac{\partial f(e^{A\xi}x)}{\partial x}e^{A\xi},\label{eq:Jaccond1}
\end{equation}

\begin{equation}
Af(x)=\frac{\partial f(x)}{\partial x}Ax.\label{eq:Jaccond2}
\end{equation}
 The latter conditions and the change of variables $y(t)=e^{A\alpha_{0}t}w(\beta_{0}t)$
in (\ref{var}) lead to
\begin{equation}
\frac{dw}{d\psi}=B_{0}(\psi)w,\qquad B_{0}(\psi)=\frac{1}{\beta_{0}}\left(\frac{\partial f(x_{0}(\psi))}{\partial x}-A\alpha_{0}\right),\label{var3}
\end{equation}
 where $\psi=\beta_{0}t$. The linear periodic system (\ref{var3})
has two periodic solutions
\begin{equation}
q_{1}(\psi)=\frac{dx_{0}(\psi)}{d\psi}, \quad q_{2}(\psi)=Ax_{0}(\psi), \label{eq:qq}
\end{equation}
 see (\ref{eq:2sol}).

Let $\mathfrak{Z}$ be the trivial vector bundle $\mathfrak{Z}=\left(\mathbb{R}^{n}\times\mathbb{T}_{1},\mathbb{T}_{1},\rho\right)$,
where $\rho$ is the natural projection onto $\mathbb{T}_{1}$. Consider
corresponding to (\ref{var3}) linear skew-product flow with time
$\psi$ and $x_{0}\in\mathbb{R}^{n}$, $\psi_{0}\in\mathbb{T}_{1}$
\begin{equation}
\pi(\psi,x_{0},\psi_{0})=\left(\Omega(\psi,\psi_{0})x_{0},\psi+\psi_{0}\right),\label{eq:*}
\end{equation}
 where $\Omega(\psi,\psi_{0})$ is the fundamental solution of (\ref{var3})
such that $\Omega(\psi_{0},\psi_{0})=I$. The vector bundle $\mathfrak{Z}$
is the sum of two sub-bundles $\mathfrak{Z}_{1}$ and $\mathfrak{Z}_{2}$,
which are invariant with respect to the flow (\ref{eq:*}). The two-dimensional
bundle $\mathfrak{Z}_{1}$ consists of periodic solutions of (\ref{var3})
spanned by two linearly independent periodic solutions (\ref{eq:qq}).
The solutions from the complementary bundle $\mathfrak{Z}_{2}$ tend
exponentially to zero as $t\to\infty.$ Since the bundle $\mathfrak{Z}_{1}$
is trivial, the bundle $\mathfrak{Z}_{2}$ is stably trivial. By \cite[p. 117]{Husemoller1993},
any stably trivial vector bundle whose fiber dimension exceeds its
base dimension is trivial. Therefore, if $n>3$, the $(n-2)$-dimensional
bundle $\mathfrak{Z}_{2}$ is trivial and there exists a smooth map
$\Phi_{0}:\mathbb{T}_{1}\to\mathcal{L}(\mathbb{R}^{n-2},\mathbb{R}^{n})$,
which is isomorphism between $\mathfrak{Z}_{2}$ and $\mathbb{T}_{1}\times\mathbb{R}^{n-2}.$
In the case $n=3$, $\Phi_{0}(\psi)$ is a $3-$dimensional vector
function, whose existence can be shown by direct analysis.

By construction, $n\times n$-matrix
\begin{equation}
\Phi_{1}(\psi)=\{q_{1}(\psi),q_{2}(\psi),\Phi_{0}(\psi)\}\label{eq:phiphi}
\end{equation}
 is non-degenerate for all $\psi.$ By the change of variables $w=\Phi_{1}(\psi)z$
system (\ref{var3}) is transformed to system
\begin{equation}
\frac{dz}{d\psi}=\mathrm{diag}\left\{ 0,0,\frac{H(\psi)}{\beta_{0}}\right\}z,\label{var4}
\end{equation}
 where $(n-2)\times(n-2)$ matrix $H(\psi)$ is $2\pi$-periodic in
$\psi$ and subsystem
\[
\frac{dz_{2}}{d\psi}=\frac{1}{\beta_{0}}H(\psi)z_{2}
\]
 is exponentially stable.

Since linear periodic system (\ref{var3}) has two nonzero linearly
independent periodic solutions, the adjoint system
\begin{equation}
\frac{dw}{d\psi}=-B_{0}^{T}(\psi)w\label{adjoint}
\end{equation}
 has also two nonzero linearly independent periodic solutions $p_{10}(\psi)$
and $p_{20}(\psi)$ \cite{Demidovich1967}. Corresponding to (\ref{adjoint})
linear skew-product flow on $\mathfrak{Z}=\left(\mathbb{R}^{n}\times\mathbb{T}_{1},\mathbb{T}_{1},\rho\right)$
has two invariant sub-bundles $\tilde{\mathfrak{Z}}_{1}$ and $\tilde{\mathfrak{Z}}_{2}$.
Two-dimensional bundle $\tilde{\mathfrak{Z}}_{1}$ consists of periodic
solutions of (\ref{adjoint}) and solutions from $\tilde{\mathfrak{Z}}_{2}$
tend exponentially to zero as $t\to-\infty.$ As in the case of system
(\ref{var3}), for the adjoint system (\ref{var4}) there exists a
smooth map $\tilde{\Phi}_{00}:\mathbb{T}_{1}\to\mathcal{L}(\mathbb{R}^{n-2},\mathbb{R}^{n})$,
which is isomorphism between $\mathfrak{\tilde{Z}}_{2}$ and $\mathbb{T}_{1}\times\mathbb{R}^{n-2}.$
Therefore, the matrix
\[
\tilde{\Phi}_{1}(\psi)=\{p_{10}(\psi),p_{20}(\psi),\tilde{\Phi}_{00}(\psi)\},
\]
 is non-degenerate for all $\psi$.

Let us show that the following holds

\begin{equation}
\tilde{\Phi}_{1}^{T}(\psi)\Phi_{1}(\psi)=\mathrm{diag}\left\{ C_{1},C_{2}(\psi)\right\} ,\label{eq:phiphit}
\end{equation}
 where $C_{1}$ is a constant non-degenerate $2\times2$ matrix and
$C_{2}(\psi)$ is a $(n-2)\times(n-2)$ non-degenerate periodic matrix.
Indeed, the scalar product of any solution of (\ref{var3}) and any
solution of its adjoint system is constant for all $\psi$ \cite{Demidovich1967}.
Hence, $p_{i0}^{T}(\psi)q_{j}(\psi)=\mathrm{const}$. Moreover, the
scalar product $\tilde{w}^{T}(\psi)w(\psi)$ of any solution $w(\psi)$
from $\mathfrak{Z}_{1}$ of system (\ref{var3}) and a solution $\tilde{w}(\psi)$
from $\mathfrak{\tilde{Z}}_{2}$ of its adjoint system is zero for
all $\psi$, since $w(\psi)$ is periodic and $\tilde{w}(\psi)\to0$
for $\psi\to-\infty$. Therefore, for all $\psi$, the fibers $\mathfrak{Z}_{1}(\psi)$
and $\tilde{\mathfrak{Z}}_{2}(\psi)$ over point $\psi$ are orthogonal
as subspaces of $\mathbb{R}^{n}$. Similarly, the subspaces $\mathfrak{\tilde{Z}}_{1}(\psi)$
and $\mathfrak{Z}_{2}(\psi)$ are orthogonal. This leads to the block-diagonal
form of the product (\ref{eq:phiphit}).

From the equality (\ref{eq:phiphit}) follows that $\Phi_{1}^{-1}(\psi)=\mathrm{diag}\left\{ C_{1}^{-1},C_{2}^{-1}(\psi)\right\} \tilde{\Phi}_{1}^{T}(\psi).$
Therefore, the matrix $\Phi_{1}^{-1}(\psi)$ has the form
\begin{equation}
\Phi_{1}^{-1}(\psi)=\{p_{1}(\psi),p_{2}(\psi),\tilde{\Phi}_{0}(\psi)\}^{T},\label{eq:phi-1-1}
\end{equation}
 where $p_{1}(\psi)$ and $p_{2}(\psi)$ are periodic solutions of
(\ref{adjoint}) such that $p_{1}^{T}(\psi)q_{1}(\psi)=p_{2}^{T}(\psi)q_{2}(\psi)=1$
and $p_{1}^{T}(\psi)q_{2}(\psi)=p_{1}^{T}(\psi)q_{2}(\psi)=0$ for
all $\psi$.

Taking into account (\ref{var4}), it can be verified that
\[
\frac{d\Phi_{1}(\psi)}{d\psi}+\Phi_{1}(\psi)\mathrm{diag}\left\{0,0,\frac{H(\psi)}{\beta_{0}}\right\}=B_{0}(\psi)\Phi_{1}(\psi).
\]
 Then the $n\times(n-2)$-matrix $\Phi_{0}(\psi)$ satisfies the relation
\begin{eqnarray}
\frac{d\Phi_{0}(\psi)}{d\psi}+\frac{1}{\beta_{0}}\Phi_{0}(\psi)H(\psi) & = & B_{0}(\psi)\Phi_{0}(\psi).\label{relation1}
\end{eqnarray}

\section{\label{sec:Local-coordinates}Local coordinates}

In this section we write systems (\ref{av01}) and (\ref{av2}), which
appear after the averaging transformation, in the local coordinates
in the vicinity of the invariant torus. Systems (\ref{av01}) and
(\ref{av2}) have the form

\begin{equation}
\frac{dx}{dt}=f(x)+\mu^{2}\bar{g}(x,\beta t)+\mu^{2}\varepsilon^{2}r_{1}(x,\beta t,\alpha t,\mu,\varepsilon)+\mu\varepsilon^{3}r_{2}(x,\beta t,\alpha t,\mu,\varepsilon),\label{av010}
\end{equation}
where $f$ and $\bar{g}$ are $S^{1}-$equivariant, i.e. $f(e^{A\varphi}x)=e^{A\varphi}f(x)$
and $\bar{g}(e^{A\varphi}x,\beta t)=e^{A\varphi}\bar{g}(x,\beta t).$
In particular, system (\ref{av01}) can be obtained from (\ref{av010}) by setting $\gamma=\mu^{2},1/\alpha=\varepsilon^{2}$,
$\bar{g}=g_{1}$, and $r_{2}\equiv0$. System (\ref{av2}) has the
form (\ref{av010}) if $\gamma^{2}/\alpha=\mu^{2},1/\alpha=\varepsilon^{2}$,
and $\bar{g}=g_{2}$, respectively. Here the perturbation terms $\mu^{2}\varepsilon^{2}r_{1}$
and $\mu\varepsilon^{3}r_{2}$ are defined from (\ref{av01}) and
(\ref{av2}) in a straightforward way. In this and the following section,
we consider $\mu$ and $\varepsilon$ as independent parameters.

We introduce new coordinates $\psi,\varphi$ and $h$ instead of $x$
in the neighborhood of the two-dimensional torus $\mathcal{T}_{2}$
by the formula
\begin{equation}
x=e^{A\varphi}\left(x_{0}(\psi)+\Phi_{0}(\psi)h\right),\label{eq:zamina}
\end{equation}
where $h\in\mathbb{R}^{n-2}$, $\|h\|\le h_{0}$ and the matrix $\Phi_{0}(\psi)$
was defined by the trivialization of bundle $\mathfrak{Z}_{2}$, see
(\ref{eq:phiphi}).

Substituting (\ref{eq:zamina}) into (\ref{av010}), we obtain
\begin{multline}
e^{A\varphi}\left(\frac{dx_{0}(\psi)}{d\psi}+\frac{d\Phi_{0}(\psi)h}{d\psi}\right)\frac{d\psi}{dt}+e^{A\varphi}\left(Ax_{0}(\psi)+A\Phi_{0}(\psi)h\right)\frac{d\varphi}{dt}+e^{A\varphi}\Phi_{0}(\psi)\frac{dh}{dt}\\
=f(e^{A\varphi}(x_{0}(\psi)+\Phi_{0}(\psi)h))+\mu^{2}\bar{g}(e^{A\varphi}(x_{0}(\psi)+\Phi_{0}(\psi)h),\beta t)\\
+\mu^{2}\varepsilon^{2}r_{1}(e^{A\varphi}(x_{0}(\psi)+\Phi_{0}(\psi)h),\beta t,\alpha t,\mu,\varepsilon)\\
+\mu\varepsilon^{3}r_{2}(e^{A\varphi}(x_{0}(\psi)+\Phi_{0}(\psi)h),\beta t,\alpha t,\mu,\varepsilon).\label{zamina2}
\end{multline}
 The following transformations allow to split the variables $h,\psi$,
and $\varphi.$ Let us write (\ref{zamina2}) in the form
\begin{multline}
\left(\frac{dx_{0}(\psi)}{d\psi}+\frac{d\Phi_{0}(\psi)h}{d\psi}\right)\left(\frac{d\psi}{dt}-\beta_{0}\right)\\
+A\left(x_{0}(\psi)+\Phi_{0}(\psi)h\right)\left(\frac{d\varphi}{dt}-\alpha_{0}\right)+\Phi_{0}(\psi)\left(\frac{dh}{dt}-H(\psi)h\right)\\
=e^{-A\varphi}f(e^{A\varphi}(x_{0}(\psi)+\Phi_{0}(\psi)h))+\mu^{2}e^{-A\varphi}\bar{g}(e^{A\varphi}(x_{0}(\psi)+\Phi_{0}(\psi)h),\beta t)\\
-\left(\frac{dx_{0}(\psi)}{d\psi}+\frac{d\Phi_{0}(\psi)h}{d\psi}\right)\beta_{0}-A\left(x_{0}(\psi)+\Phi_{0}(\psi)h\right)\alpha_{0}-\Phi_{0}(\psi)H(\psi)h\\
+\mu^{2}\varepsilon^{2}e^{-A\varphi}r_{1}(e^{A\varphi}(x_{0}(\psi)+\Phi_{0}(\psi)h),\beta t,\alpha t,\mu,\varepsilon)\\
+\mu\varepsilon^{3}e^{-A\varphi}r_{2}(e^{A\varphi}(x_{0}(\psi)+\Phi_{0}(\psi)h),\beta t,\alpha t,\mu,\varepsilon).\label{zamina22}
\end{multline}
 Taking into account (\ref{riv1}) and (\ref{relation1}) we have
\begin{multline}
\left(\Phi_{1}(\psi)+\Phi_{2}(\psi,h)\right)\left(\begin{array}{c}
\frac{d\psi}{dt}-\beta_{0}\\
\frac{d\varphi}{dt}-\alpha_{0}\\
\frac{dh}{dt}-H(\psi)h
\end{array}\right)=f(x_{0}+\Phi_{0}h)-f(x_{0})-\frac{\partial f(x_{0})}{\partial x}\Phi_{0}h\\
+\mu^{2}\bar{g}(x_{0},\beta t)+\mu^{2}\tilde{g}(h,\psi,\beta t)h+\mu^{2}\varepsilon^{2}r_{3}(h,\psi,\varphi,\beta t,\alpha t,\mu,\varepsilon)+\mu\varepsilon^{3}r_{4}(h,\psi,\varphi,\beta t,\alpha t,\mu,\varepsilon),\label{zam2}
\end{multline}
 where
\begin{gather*}
\Phi_{1}(\psi)=\left[\frac{dx_{0}(\psi)}{d\psi}\ \,\, Ax_{0}(\psi)\ \,\,\Phi_{0}(\psi)\right],\\
\Phi_{2}(\psi,h)=\left[\frac{d\Phi_{0}(\psi)}{d\psi}h\,\,\ A\Phi_{0}(\psi)h\ \,\,0\right],\\
\tilde{g}(h,\psi,\beta t)h=\bar{g}(x_{0}(\psi)+\Phi_{0}(\psi)h,\beta t)-\bar{g}(x_{0}(\psi),\beta t).
\end{gather*}
 Since by our construction $\det\Phi_{1}(\psi)\neq0$ for all $\psi\in\mathbb{T}_{1},$
the matrix $\Phi_{1}(\psi)+\Phi_{2}(\psi,h)$ is invertible for sufficiently
small $h$ and
\[
(\Phi_{1}(\psi)+\Phi_{2}(\psi,h))^{-1}=\Phi_{1}^{-1}(\psi)+\tilde{H}(h,\psi)=\left(\begin{array}{c}
p_{1}^{T}(\psi)\\
p_{2}^{T}(\psi)\\
\tilde{\Phi}_{0}^{T}(\psi)
\end{array}\right)+\left(\begin{array}{c}
\tilde{H}_{1}(h,\psi)\\
\tilde{H}_{2}(h,\psi)\\
\tilde{H}_{3}(h,\psi)
\end{array}\right),
\]
 where $C^{l-3}$-smooth function $\tilde{H}(h,\psi)=\mathcal{O}(\|h\|)$
is periodic in $\psi.$

Therefore, system (\ref{zam2}) can be solved with respect to derivatives
$d\psi/dt,$ $d\varphi/dt$ and $dh/dt:$
\begin{eqnarray}
\frac{d\psi}{dt} & = & \beta_{0}+\mu^{2}p_{1}^{T}(\psi)\bar{g}(x_{0},\beta t)+R_{11}+\mu^{2}R_{12}+\mu^{2}\varepsilon^{2}R_{13}+\mu\varepsilon^{3}R_{14},\label{ss1}\\[3mm]
\frac{d\varphi}{dt} & = & \alpha_{0}+\mu^{2}p_{2}^{T}(\psi)\bar{g}(x_{0},\beta t)+R_{21}+\mu^{2}R_{22}+\mu^{2}\varepsilon^{2}R_{23}+\mu\varepsilon^{3}R_{24},\label{ss2}\\[3mm]
\frac{dh}{dt} & = & H(\psi)h+R_{01}+\mu^{2}R_{02}+\mu^{2}\varepsilon^{2}R_{03}+\mu\varepsilon^{3}R_{04},\label{ss3}
\end{eqnarray}
 where
\begin{eqnarray*}
 &  & R_{j1}=R_{j1}(h,\psi,\mu)=(p_{j}^{T}(\psi)+\tilde{H}_{j}(h,\psi))F_{1}(h,\psi),\\[2mm]
 &  & R_{j2}=R_{j2}(h,\psi,\beta t,\mu)=(p_{j}^{T}(\psi)+\tilde{H}_{j}(h,\psi))\tilde{g}(h,\psi,\beta t)h+\tilde{H}_{j}(h,\psi)\bar{g}(x_{0},\beta t),\\[2mm]
 &  & R_{jk}=R_{jk}(h,\psi,\varphi,\beta t,\alpha t,\mu,\varepsilon)=(p_{j}^{T}(\psi)+\tilde{H}_{j}(h,\psi))r_{k}(h,\psi,\varphi,\beta t,\alpha t,\mu,\varepsilon),\\[2mm]
 &  & R_{01}=R_{01}(h,\psi,\mu)=(\tilde{\Phi}_{0}(\psi)+\tilde{H}_{3}(h,\psi))F_{1}(h,\psi),\\[2mm]
 &  & R_{02}=R_{02}(h,\psi,\beta t,\mu)=(\tilde{\Phi}_{0}(\psi)+\tilde{H}_{3}(h,\psi))\bar{g}(x_{0}(\psi)+\Phi_{0}(\psi)h,\beta t),\\[2mm]
 &  & R_{0k}=R_{03}(h,\psi,\varphi,\beta t,\alpha t,\mu,\varepsilon)=(\tilde{\Phi}_{0}(\psi)+\tilde{H}_{3}(h,\psi))r_{k}(h,\psi,\varphi,\beta t,\alpha t,\mu,\varepsilon),\\[2mm]
 &  & F_{1}(h,\psi)=f(x_{0}+\Phi_{0}h)-f(x_{0})-\frac{\partial f(x_{0})}{\partial x}\Phi_{0}h=\mathcal{O}(\|h\|^{2})
\end{eqnarray*}
 with $j=1,2$, $k=3,4$.

The following lemma establishes the existence of the perturbed manifold.
Note that the similar lemma has been proved in \cite{Recke2011} for
the system with constant matrix $H$ and somewhat different parameter
dependences.

\begin{lemma}
\label{lemma2} For $\mu\in(0,\mu_{0})$ and $\varepsilon\in(0,\varepsilon_{0})$
with sufficiently small $\mu_{0}$ and $\varepsilon_{0}$, the system
(\ref{ss1}) -- (\ref{ss3}) has an integral manifold
\begin{equation}
\mathfrak{M}_{\mu,\varepsilon}=\{(h,\psi,\varphi,t):\, h=u(\psi,\varphi,\beta t,\alpha t,\mu,\varepsilon),\,(\psi,\varphi)\in\mathbb{T}_{1}\times\mathbb{T}_{1},t\in\mathbb{R}\},\label{mani}
\end{equation}
 where the function $u$ has the form
\begin{multline}
u(\psi,\varphi,\beta t,\alpha t,\mu,\varepsilon)=\mu^{2}u_{0}(\psi,\beta t,\mu)\\
+\varepsilon^{3}\mu u_{1}(\psi,\varphi,\beta t,\alpha t,\mu,\varepsilon)+\varepsilon^{2}\mu^{2}u_{2}(\psi,\varphi,\beta t,\alpha t,\mu,\varepsilon).\label{mani00}
\end{multline}
 Here functions $u_{0},u_{1}$ and $u_{2}$ are $C^{l-3}$-smooth,
$2\pi$-periodic in $\psi,\varphi,\beta t$ and $\alpha t$, satisfying
$\|u_{j}\|_{C^{l-3}}\le M_{1},\ j=0,1,2,$ with positive constant
$M_{1}$, which does not depend on $\mu,\varepsilon,$ and $\alpha\ge\alpha_{0}$.
Here $\|.\|_{C^{l-3}}$ is the norm of functions from $C^{l-3}(\mathbb{T}_{4})$
with fixed parameters $\mu$ and $\varepsilon.$

The integral manifold $\mathfrak{M}_{\mu,\varepsilon}$ is asymptotically
stable, i.e. there exists $\nu_{0}=\nu_{0}(\mu,\varepsilon_{0})$
such that for every initial value $(h,\psi,\varphi)$ at time $\tau$
with $\|h\|\le\nu_{0}$ there exists a unique $(\psi_{0},\varphi_{0})$
such that
\begin{multline}
\|N(t,\tau,h,\psi,\varphi)-N(t,\tau,u(\psi_{0},\varphi_{0},\beta\tau,\alpha\tau,\mu,\varepsilon),\psi_{0},\varphi_{0})\|\\
\le Le^{-\kappa(t-\tau)}\|(h,\psi,\varphi)-(u(\psi_{0},\varphi_{0},\beta\tau,\alpha\tau,\mu,\varepsilon),\psi_{0},\varphi_{0})\|,\label{mani2}
\end{multline}
 where $t\ge\tau,$ $L$ and $\kappa$ are some positive constants
not depending on $\alpha,\mu,\varepsilon,$ $N(t,\tau,h,\psi,\varphi)$
is the solution of the system (\ref{ss1}) -- (\ref{ss3}) with an
initial value $N(\tau,\tau,h,\psi,\varphi)=(h,\psi,\varphi)$.
\end{lemma}
\begin{proof}
By introducing new variables $\zeta=(\psi,\varphi,\zeta_{1},\zeta_{2}),$
$\zeta_{1}=\beta t,$ $\zeta_{2}=\alpha t,$ and new parameters $\lambda=(\eta_{1},\eta_{2},\eta_{3},\mu,\varepsilon)$,
$\eta_{1}=\mu^{2},$ $\eta_{2}=\varepsilon^{2}\mu^{2},$ $\eta_{3}=\varepsilon^{3}\mu$
in system (\ref{ss1}) -- (\ref{ss3}), we obtain the following autonomous
system
\begin{eqnarray}
\frac{dh}{dt} & = & H(\psi)h+\tilde{Q}_{0}(h,\zeta,\lambda),\label{ss1-ext}\\[3mm]
\frac{d\zeta}{dt} & = & \omega_{0}+\tilde{Q}(h,\zeta,\lambda),\label{ss2-ext}
\end{eqnarray}
 where $\omega_{0}=(\beta_{0},\alpha_{0},\beta,\alpha)$ and $\tilde{Q}=(\tilde{Q}_{1},\tilde{Q}_{2},0,0),$
\begin{eqnarray*}
 &  & \tilde{Q}_{0}=R_{01}+\eta_{1}R_{02}+\eta_{2}R_{03}+\eta_{3}R_{04},\\
 &  & \tilde{Q}_{j}=\eta_{1}p_{j}^{T}\bar{g}+R_{j1}+\eta_{1}R_{j2}+\eta_{2}R_{j3}+\eta_{3}R_{j4},\,\, j=1,2.
\end{eqnarray*}
 The corresponding reduced system has the form
\begin{equation}
\frac{dh}{dt}=H(\psi)h,\quad\frac{d\zeta}{dt}=\omega_{0}.\label{eq:redsys}
\end{equation}
 By construction, this system is exponentially stable, i.e. for all
$\psi\in\mathbb{T}_{1}$ the fundamental solution $\Omega_{0}(t,\tau,\psi)$
of system $\frac{dh}{dt}=H(\beta_{0}t+\psi)h$ satisfies
\begin{equation}
\|\Omega_{0}(t,\tau,\psi)\|\le\mathcal{L}e^{-\kappa_{0}(t-\tau)},\ t\ge\tau,\label{1**}
\end{equation}
 where $\mathcal{L}\ge1$ and $\kappa_{0}>0.$ Note that matrix $H(\psi)$
and the fundamental solution $\Omega_{0}(t,\tau,\psi)$ depend only
on 
the first coordinate of vector $\zeta=(\psi,\varphi,\zeta_{1},\zeta_{2})$.

Denote $I_{\lambda_{0}}=\{\lambda:\ \|\lambda\|\le\lambda_{0}\}.$
Let $\mathcal{F}_{\rho}$ be the space of Lipschitz continuous functions
$w:\,\mathbb{T}_{4}\times I_{\lambda_{0}}\to\mathbb{R}^{n}$ such
that $\|w\|_{C}\le\rho,$ ${\rm \mathrm{Lip}_{\zeta}}\, w\le\rho$
for all $\lambda\in I_{\lambda_{0}}$, where ${\rm \mathrm{Lip}_{\zeta}}\, w$
is the Lipschitz constant of $w$ with respect to the first argument
$\zeta$.

In order to prove the existence of the invariant manifold for (\ref{ss1-ext})
-- (\ref{ss2-ext}), we consider the mapping $T:\mathcal{F}_{\rho}\to\mathcal{F}_{\rho}$
\[
T(w)(\zeta,\lambda)=\int_{-\infty}^{0}\Omega(\tau,\zeta)\tilde{Q}_{0}\left(w\left(\zeta_{\tau},\lambda\right),\zeta_{\tau},\lambda\right)d\tau,
\]
 where $\zeta_{\tau}=(\psi_{\tau},\varphi_{\tau},\zeta_{1\tau},\zeta_{2\tau})$
is the solution of (\ref{ss2-ext}) for $h=w(\zeta,\lambda)$ with
initial condition $\zeta,$ i.e.
\[
\frac{d\zeta}{dt}=\omega_{0}+\tilde{Q}(w(\zeta,\lambda),\zeta,\lambda).
\]
 $\Omega(t,\zeta)$ is the fundamental solution of system
\begin{equation}
\frac{dh}{dt}=H(\psi_{t})h,\label{lin-w}
\end{equation}
 with $\Omega(0,\zeta)=I$. As follows from \cite{Yi1993a} (Lemma
6.3), system (\ref{lin-w}) is exponentially stable as a small perturbation
of system (\ref{eq:redsys}), therefore there exist $\lambda_{0}>0$ and
$\rho_{0}>0$ such that for $\lambda\in I_{\lambda_{0}}$ and $w\in\mathcal{F}_{\rho_{0}}$
\begin{equation}
\|\Omega(t,\zeta)\|\le\mathcal{L}_{1}e^{-\kappa_{1}t},\quad t\ge0,\label{2**}
\end{equation}
 with $\mathcal{L}_{1}\ge1$ and $0<\kappa_{1}\le\kappa_{0}.$

Analogously to \cite{Recke2011} we prove that the map $T(w):\ \mathcal{F}_{\rho}\to\mathcal{F}_{\rho}$
is a contraction for small enough $\rho_{0}$ and $\lambda_{0}$.
Hence, the mapping $T(w)$ has unique fixed point $w_{0}(\zeta,\lambda)$.

For proving $C^{l-3}$ smoothness of integral manifold $w_{0}(\zeta,\lambda)$
we use the fiber contraction theorem \cite{Vanderbauwhede1989}, \cite[p. 127]{Chicone2006}.
In particular, for the proof of $C^{1}$-smoothness with respect to
$\zeta$, we introduce the set $\mathcal{F}^{1}$ of all bounded continuous
functions $\Phi$ that map $\mathbb{T}_{4}\times I_{\lambda_{0}}$
into the set of all $n\times4$ matrices. Let $\mathcal{F}_{\rho}^{1}$
denote the closed ball in $\mathcal{F}^{1}$ with radius $\rho.$
For $w\in\mathcal{F}_{\rho},$ we consider the map $T^{1}(w,\Phi):\,\mathcal{F}_{\rho}\times\mathcal{F}_{\rho}^{1}\to\mathcal{F}_{\rho}^{1}$,
which is defined as follows
\begin{eqnarray}
 &  & T^{1}(w,\Phi)(\zeta)=\int_{-\infty}^{0}\Omega(\tau,\zeta)\Biggl(\frac{\partial\tilde{Q}_{0}(w(\zeta_{\tau},\lambda),\zeta_{\tau},\lambda)}{\partial\zeta}W_{\tau}\nonumber \\
 &  & +\frac{\partial\tilde{Q}_{0}(w(\zeta_{\tau},\lambda),\zeta_{\tau},\lambda)}{\partial h}\Phi(\zeta_{\tau},\lambda)W_{\tau}+\frac{\partial H(\zeta_{\tau})}{\partial\zeta}w(\zeta_{\tau},\lambda)\Biggl)d\tau,\label{oznT1}
\end{eqnarray}
 where $\zeta_{t}$ and $W_{t}$ are solutions of the system
\begin{eqnarray}
 &  & \frac{d\zeta}{dt}=\omega_{0}+\tilde{Q}(w(\zeta,\lambda),\zeta,\lambda),\label{eq-1w}\\
 &  & \frac{dW}{dt}=\frac{\partial\tilde{Q}(w(\zeta,\lambda),\zeta,\lambda)}{\partial\zeta}W+\frac{\partial\tilde{Q}(w(\zeta,\lambda),\zeta,\lambda)}{\partial h}\Phi(\zeta,\lambda)W.\label{eq-2w}
\end{eqnarray}
 Analogously to \cite{Chicone2006,Recke2011} we prove that the map
\begin{equation}
(w,\Phi)\to(T(w),T^{1}(w,\Phi))\label{TT1}
\end{equation}
 is continuous with respect to $w$ and $\Phi$ and is a fiber contraction.
It has unique fixed point $(w_{0},w_{1})$ which is globally attracting
and bounded uniformly with respect to $\alpha\in[\alpha_{0},\infty).$
Repeating \cite[p.296]{Chicone2006} one can show that $w_{0}$ is
continuously differentiable and $Dw_{0}=w_{1}.$

The smoothness up to $C^{l-3}$ can be improved inductively. The continuous
differentiability with respect to $\lambda$ is proved analogously.

Since the invariant manifold $h=w(\zeta,\lambda)$ for $\eta_{1}=\eta_{2}=\eta_{3}=0$
equals to zero $h=0$, it can be represented as
\[
h=\eta_{1}w_{0}(\psi,\zeta_{1},\mu)+\eta_{2}w_{1}(\psi,\varphi,\zeta_{1},\zeta_{2},\lambda)+\eta_{3}w_{2}(\psi,\varphi,\zeta_{1},\zeta_{2},\lambda).
\]
 Note that $w_{0}$ does not depend on $\zeta_{2}$, $\eta_{2},\eta_{3}$,
and $\varepsilon$, since system (\ref{ss1-ext}) -- (\ref{ss2-ext})
is independent on $\zeta_{2}$ for $\eta_{2}=\eta_{3}=\varepsilon=0$.
Taking into account the definitions of $\eta_{1},$ $\eta_{2}$, $\eta_{3},$
and $\zeta_{1}$, $\zeta_{2}$, we obtain that the invariant manifold
of (\ref{ss1}) -- (\ref{ss3}) has the form (\ref{mani00}).

The remaining proof of (\ref{mani2}) is analogous to the proof of
Lemma 5.1 in \cite{Recke2011}.
\end{proof}

\begin{remark}
{\rm If $r_{2}\equiv0$ in system (\ref{av010}) then $R_{04}\equiv0,$
$R_{14}\equiv R_{24}\equiv0$ in system (\ref{ss1}) -- (\ref{ss3})
and $u_{1}\equiv0$ in the expression for the integral manifold (\ref{mani00}).
This corresponds to the non-degenerate case described by (\ref{av01})
with non-vanishing averaged term $g_{1}(x_{1},\beta t)$.}
\end{remark}

\section{Investigation of the system on the manifold\label{sec:Investigation-of-the}}

Substituting the expression for the invariant manifold (\ref{mani00})
into equations (\ref{ss1}) -- (\ref{ss3}), we obtain the system
on the manifold
\begin{eqnarray}
 &  & \frac{d\psi}{dt}=\beta_{0}+\mu^{2}p_{1}^{T}(\psi)\bar{g}(x_{0}(\psi),\beta t)+\mu^{4}S_{11}(\psi,\beta t,\mu)+\nonumber \\[2mm]
 &  & \hspace{35mm}+\varepsilon^{2}\mu^{2}S_{12}(\psi,\varphi,\beta t,\alpha t,\mu,\varepsilon)+\varepsilon^{3}\mu S_{13}(\psi,\varphi,\beta t,\alpha t,\mu,\varepsilon),\label{m1}\\[2mm]
 &  & \frac{d\varphi}{dt}=\alpha_{0}+\mu^{2}p_{2}^{T}(\psi)\bar{g}(x_{0}(\psi),\beta t)+\mu^{4}S_{21}(\psi,\beta t,\mu)+\nonumber \\[3mm]
 &  & \hspace{35mm}+\varepsilon^{2}\mu^{2}S_{22}(\psi,\varphi,\beta t,\alpha t,\mu,\varepsilon)+\varepsilon^{3}\mu S_{23}(\psi,\varphi,\beta t,\alpha t,\mu,\varepsilon),\label{m2}
\end{eqnarray}
 where $C^{l-3}$-smooth functions $S_{jl}$ are $2\pi$-periodic
in $\psi,\varphi,\beta t,$ and $\alpha t.$

Now we will use the closeness of the frequencies $\beta_{0}$ and
$\beta$
\[
\beta-\beta_{0}=\mu^{2}\Delta
\]
 and obtain conditions for the locking of the variable $\psi$ to
the external frequency $\beta$. For this, we introduce new variable
$\psi_{1}$ in (\ref{m1}) -- (\ref{m2}) according to the formula
\[
\psi=\beta t+\psi_{1}.
\]
 In the obtained system
\begin{eqnarray}
 &  & \frac{d\psi_{1}}{dt}=-\mu^{2}\Delta+\mu^{2}p_{1}^{T}(\beta t+\psi_{1})\bar{g}(x_{0}(\beta t+\psi_{1}),\beta t)+\mu^{4}S_{11}(\beta t+\psi_{1},\beta t,\mu)+\nonumber \\[2mm]
 &  & \hspace{25mm}+\varepsilon^{2}\mu^{2}S_{12}(\beta t+\psi_{1},\varphi,\beta t,\alpha t,\mu,\varepsilon)+\varepsilon^{3}\mu S_{13}(\beta t+\psi_{1},\varphi,\beta t,\alpha t,\mu,\varepsilon),\label{eq:theta-1}\\[2mm]
 &  & \frac{d\varphi}{dt}=\alpha_{0}+\mu^{2}p_{2}^{T}(\beta t+\psi_{1})\bar{g}(x_{0}(\beta t+\psi_{1}),\beta t)+\mu^{4}S_{21}(\beta t+\psi_{1},\beta t,\mu)+\nonumber \\[3mm]
 &  & \hspace{25mm}+\varepsilon^{2}\mu^{2}S_{22}(\beta t+\psi_{1},\varphi,\beta t,\alpha t,\mu,\varepsilon)+\varepsilon^{3}\mu S_{23}(\beta t+\psi_{1},\varphi,\beta t,\alpha t,\mu,\varepsilon)\label{eq:phi-1}
\end{eqnarray}
 the variable $\psi_{1}$ is slow and one can perform the following
averaging transformation with respect to $\beta t$
\[
\psi_{1}=\psi_{2}+\frac{\mu^{2}}{\beta}\int_{0}^{\beta t}[p_{1}^{T}(\xi+\psi_{1})\bar{g}(x_{0}(\xi+\psi_{1}),\xi)-\bar{G}(\psi_{1})]d\xi,
\]

\[
\varphi=\varphi_{2}+\frac{\mu^{2}}{\beta}\int_{0}^{\beta t}[p_{2}^{T}(\xi+\psi_{1})\bar{g}(x_{0}(\xi+\psi_{1}),\xi)-\bar{S}_{21}(\psi_{1})]d\xi,
\]
 where
\[
G(\psi_{1})=\frac{1}{2\pi}\int_{0}^{2\pi}p_{1}^{T}(\xi+\psi_{1})\bar{g}(x_{0}(\xi+\psi_{1}),\xi)d\xi,
\]

\[
\bar{S}_{21}(\psi_{1})=\frac{1}{2\pi}\int_{0}^{2\pi}p_{2}^{T}(\xi+\psi_{1})\bar{g}(x_{0}(\xi+\psi_{1}),\xi)d\xi.
\]
 After this transformation system (\ref{eq:theta-1}) -- (\ref{eq:phi-1-1})
takes the form
\begin{eqnarray}
 &  & \frac{d\psi_{2}}{dt}=-\Delta\mu^{2}+\mu^{2}G(\psi_{2})+\mu^{4}\tilde{S}_{11}(\psi_{2},\beta t,\mu)\label{mm1}\\[2mm]
 &  & \,\,\,\,\,\,\,\,\,\,\,\,\,\,\,\,\,\,\,+\varepsilon^{2}\mu^{2}\tilde{S}_{12}(\psi_{2},\varphi_{2},\beta t,\alpha t,\mu,\varepsilon)+\varepsilon^{3}\mu\tilde{S}_{13}(\psi_{2},\varphi_{2},\beta t,\alpha t,\mu,\varepsilon),\nonumber \\
 &  & \frac{d\varphi_{2}}{dt}=\alpha_{0}+\mu^{2}\bar{S}_{21}(\psi_{2})+\mu^{4}\tilde{S}_{21}(\psi_{2},\beta t,\mu)\label{mm2}\\
 &  & \,\,\,\,\,\,\,\,\,\,\,\,\,\,\,\,\,\,\,+\varepsilon^{2}\mu^{2}\tilde{S}_{22}(\psi_{2},\varphi_{2},\beta t,\alpha t,\mu,\varepsilon)+\varepsilon^{3}\mu\tilde{S}_{23}(\psi_{2},\varphi_{2},\beta t,\alpha t,\mu,\varepsilon),
\end{eqnarray}
 where the functions in the right hand side are $C^{l-3}$-smooth
and periodic in $\psi_{2},\varphi_{2},\beta t,$ and $\alpha t.$

The corresponding averaged system allows splitting off of the dynamics
for $\psi_{2}$ variable
\begin{equation}
\frac{d\psi_{2}}{dt}=\mu^{2}\left(-\Delta+G(\psi_{2})\right),\label{mm01}
\end{equation}
 which can be described completely. In particular, it has equilibriums,
which are solutions of the equation
\begin{equation}
\Delta=G(\xi).\label{alg}
\end{equation}
 Such equilibriums exist if
\[
\Delta=\frac{\beta-\beta_{0}}{\mu^{2}}\in[G^{-},G^{+}],
\]
 where $G^{-}$ and $G^{+}$ are minimum and maximum of the periodic
function $G(\xi)$, respectively.

If $\Delta\in[G^{-},G^{+}]$ is a regular value of the map $G$, then
there exist even number of solutions $\xi=\vartheta_{j}^{0}$, $j=1,\dots,2\tilde{N}$,
of equation (\ref{alg}) such that $G'(\vartheta_{j}^{0})\neq0.$
In general, $\tilde{N}$ depends on $\Delta$. Since the signs of
every two sequential values $G'(\vartheta_{j}^{0})$ and $G'(\vartheta_{j+1}^{0})$
are opposite, the half of these equilibriums are stable and another
half is unstable. Every such equilibrium corresponds to an integral
manifold of (\ref{mm1})--(\ref{mm2}) with the same stability properties.

\begin{lemma}
\label{lemma4} Assume $\Delta\in[G^{-},G^{+}]$ such that (\ref{alg})
has $2\tilde{N}$ nondegenerate solutions $\xi=\vartheta_{j}^{0}$,
$j=1,\dots,2\tilde{N}$. Then there exist $\mu_{0}>0$ and $c_{0}>0$
such that for all $0<\mu\le\mu_{0}$ and $\varepsilon\le c_{0}\mu^{1/3}$
system (\ref{mm1}) -- (\ref{mm2}) has $2\tilde{N}$ integral manifolds
\begin{equation}
\Pi_{j}=\left\{ \left(\psi_{2},\varphi_{2},t\right):\ \psi_{2}=\vartheta_{j}^{0}+v_{j}(\varphi_{2},\beta t,\alpha t,\mu,\varepsilon),\,\varphi_{2}\in\mathbb{T}_{1},t\in\mathbb{R}\right\} ,\label{ma}
\end{equation}
 where
\[
v_{j}=\mu^{2}v_{j0}(\beta t,\mu)+\varepsilon^{2}v_{j1}(\varphi_{2},\beta t,\alpha t,\mu,\varepsilon)+\frac{\varepsilon^{3}}{\mu}v_{j2}(\varphi_{2},\beta t,\alpha t,\mu,\varepsilon),
\]
 with $C^{l-3}$ smooth, periodic in $\varphi_{2},\beta t,\alpha t$
functions $v_{jk}$, such that $\left\Vert v_{jk}\right\Vert _{C^{l-3}}\le M_{3}$
with the constant $M_{3}$ independent on $\mu,\varepsilon,$ and
$\alpha\ge\alpha_{0}.$

The manifolds $\Pi_{2k}$, $k=1,...,\tilde{N},$ are exponentially
stable in the following sense: there exists $\delta_{0}$ such that
if $|\psi_{20}-\vartheta_{2k}^{0}|\le\delta_{0}$ and $\varphi_{0}\in\mathbb{T}_{1}$,
then there exists a unique $\varphi_{01}$ such that for $t\ge t_{0}$
the following inequality holds
\begin{multline}
|\psi_{2}(t,t_{0},\psi_{20},\varphi_{0})-\psi_{2}(t,t_{0},\vartheta_{2k}^{0}+v_{2k}(\varphi_{01},\beta t_{0},\alpha t_{0},\mu,\varepsilon),\varphi_{01})|\\
+|\varphi_{2}(t,t_{0},\psi_{20},\varphi_{0})-\varphi_{2}(t,t_{0},\vartheta_{2k}^{0}+v_{2k}(\varphi_{01},\beta t_{0},\alpha t_{0},\mu,\varepsilon),\varphi_{01})|\\
\le\mathcal{L}_{2}e^{-\mu^{2}\kappa_{2}(t-t_{0})}\left(|\varphi_{0}-\varphi_{01}|+|\psi_{20}-\vartheta_{2k}^{0}-v_{2k}(\varphi_{01},\beta t_{0},\alpha t_{0},\mu,\varepsilon)|\right),\label{sta}
\end{multline}
 where $\mathcal{L}_{2}\ge1$ and $\kappa_{2}>0$ is independent on
$\alpha,\mu$, and $\varepsilon.$

The manifolds $\Pi_{2k-1}$, $k=1,...,\tilde{N},$ are exponentially
unstable in the following sense: there exists $\delta_{0}$ such that
if $|\psi_{20}-\vartheta_{2k-1}^{0}|\le\delta_{0}$ and $\varphi_{0}\in\mathbb{T}_{1}$,
then there exists a unique $\varphi_{01}$ such that for $t\le t_{0}$
the following inequality holds
\begin{multline}
|\psi_{2}(t,t_{0},\psi_{20},\varphi_{0})-\psi_{2}(t,t_{0},\vartheta_{2k-1}^{0}+v_{2k-1}(\varphi_{01},\beta t_{0},\alpha t_{0},\mu,\varepsilon),\varphi_{01})|\\
+|\varphi_{2}(t,t_{0},\psi_{20},\varphi_{0})-\varphi_{2}(t,t_{0},\vartheta_{2k-1}^{0}+v_{2k-1}(\varphi_{01},\beta t_{0},\alpha t_{0},\mu,\varepsilon),\varphi_{01})|\\
\le\mathcal{L}_{3}e^{\mu^{2}\kappa_{3}(t-t_{0})}\left(|\varphi_{0}-\varphi_{01}|+|\psi_{20}-\vartheta_{2k-1}^{0}-v_{2k-1}(\varphi_{01},\beta t_{0},\alpha t_{0},\mu,\varepsilon)|\right),\label{sta2-}
\end{multline}
 where $\mathcal{L}_{3}\ge1$ and $\kappa_{3}>0$ is independent on
$\alpha,\mu$, and $\varepsilon.$
\end{lemma}
\begin{proof}
Let us consider a neighborhood of the point $\psi_{2}=\vartheta_{k}^{0}$.
Setting $\zeta_{1}=\beta t,$ $\zeta_{2}=\alpha t$, new time $\tau=\mu^{2}t$,
and new parameters $\eta_{1}=\mu^{2},$ $\eta_{2}=\varepsilon^{2}$,
$\eta_{3}=\varepsilon^{3}/\mu$, $\chi=1/\mu^{2}$ in (\ref{mm1})--(\ref{mm2})
we change the variables $\psi_{2}=\vartheta_{k}^{0}+\psi_{3}$ and
obtain the following autonomous system on $4$-dimensional torus $\mathbb{T}_{4}:$
\begin{eqnarray}
 &  & \frac{d\psi_{3}}{d\tau}=G'(\vartheta_{k}^{0})\psi_{3}+\bar{G}_{2}(\psi_{3})\psi_{3}^{2}+\eta_{1}\tilde{S}_{11}(\vartheta_{k}^{0}+\psi_{3},\zeta_{1},\mu)\nonumber \\
 &  & \hspace{5mm}+\eta_{2}\tilde{S}_{12}(\vartheta_{k}^{0}+\psi_{3},\varphi_{2},\zeta_{1},\zeta_{2},\mu,\varepsilon)+\eta_{3}\tilde{S}_{13}(\vartheta_{k}^{0}+\psi_{3},\varphi_{2},\zeta_{1},\zeta_{2},\mu,\varepsilon),\label{mm1j}\\
 &  & \frac{d\varphi_{2}}{d\tau}=\alpha_{0}\chi+\bar{S}_{21}(\vartheta_{k}^{0}+\psi_{3},\mu)+\eta_{1}\tilde{S}_{21}(\vartheta_{k}^{0}+\psi_{3},\zeta_{1},\mu)\nonumber \\
 &  & \hspace{5mm}+\eta_{2}\tilde{S}_{22}(\vartheta_{k}^{0}+\psi_{3},\varphi_{2},\zeta_{1},\zeta_{2},\mu,\varepsilon)+\eta_{3}\tilde{S}_{23}(\vartheta_{k}^{0}+\psi_{3},\varphi_{2},\zeta_{1},\zeta_{2},\mu,\varepsilon),\label{mm2j}\\
 &  & \frac{d\zeta_{1}}{d\tau}=\beta\chi,\qquad\frac{d\zeta_{2}}{d\tau}=\alpha\chi,\label{mm3j}
\end{eqnarray}
 where $\bar{G}_{2}(\psi_{3})\psi_{3}^{2}:=G(\vartheta_{k}^{0}+\psi_{3})-G(\vartheta_{k}^{0})-G'(\vartheta_{k}^{0})\psi_{3}$.

In the function space $C^{l-3}\left(\mathbb{T}_{3}\times I_{\lambda_{0}}\right)$
of bounded together with their $l-3$ derivatives functions $w(\varphi_{2},\zeta_{1},\zeta_{2},\lambda)$
defined on $(\varphi_{2},\zeta_{1},\zeta_{2})\in\mathbb{T}_{3},$
$\lambda=(\eta_{1},\eta_{2},\eta_{1},\mu,\varepsilon)\in I_{\lambda_{0}}=\{\lambda:\ \|\lambda\|\le\lambda_{0}\},$
we introduce the mapping
\[
T_{k}(w)=-\int_{0}^{\infty}\exp\left(-G'(\vartheta_{k}^{0})\xi\right)Q_{4}(w(\varphi_{2\xi},\zeta_{1\xi},\zeta_{2\xi},\lambda),\varphi_{2\xi},\zeta_{1\xi},\zeta_{2\xi},\lambda)d\xi
\]
 if $G'(\vartheta_{k}^{0})>0,$ and
\[
T_{k}(w)=\int_{-\infty}^{0}\exp\left(-G'(\vartheta_{k}^{0})\xi\right)Q_{4}(w(\varphi_{2\xi},\zeta_{1\xi},\zeta_{2\xi},\lambda),\varphi_{2\xi},\zeta_{1\xi},\zeta_{2\xi},\lambda)d\xi
\]
 for $G'(\vartheta_{k}^{0})<0$. Here $Q_{4}$ is the right hand side
of (\ref{mm1j}) with exception of the linear term, $\varphi_{2\xi}=\varphi_{2}(\xi,\varphi,\zeta_{1},\zeta_{2},\lambda)$,
$\zeta_{1\xi}=\beta\xi+\zeta_{1}$, $\zeta_{2\xi}=\alpha\xi+\zeta_{2}$
is the solution of (\ref{mm2j}) -- (\ref{mm3j}) for $\psi_{3}=w(\varphi_{2},\zeta_{1},\zeta_{2},\lambda)$.

Analogously to the proof of Lemma \ref{lemma2} (see also \cite{Recke2011}),
we apply the fiber contraction theorem and show that there exists
a unique fixed point
\begin{equation}
w_{k}=\eta_{1}v_{k1}(\zeta_{1},\lambda)+\eta_{2}v_{k2}(\varphi_{2},\zeta_{1},\zeta_{2},\lambda)+\eta_{3}v_{k3}(\varphi_{2},\zeta_{1},\zeta_{2},\lambda)\label{eq:w}
\end{equation}
 of $T_{k}(w)$ in the neighborhood of $(0,0)\in C^{l-3}(\mathbb{T}_{3}\times I_{\lambda_{0}})$.
Functions in right-hand side of (\ref{eq:w}) are $C^{l-3}$ smooth
and $2\pi$-periodic in $\varphi_{2},\zeta_{1},\zeta_{2},$ such that
$\|v_{kj}\|_{C^{l-4}}\le M_{2},$ where positive constant $M_{2}$
does not depend on $\lambda,$ $\chi\ge\chi_{0}$ and $\alpha\ge\alpha_{0}.$

Respectively, there exist $\mu_{0}>0$ and $c_{0}>0$ such that for
all $0<\mu\le\mu_{0}$ and $\varepsilon\le c_{0}\mu^{1/3}$ the system
(\ref{mm1j}) -- (\ref{mm3j}) possesses $2\tilde{N}$ integral manifolds
(\ref{ma}).

The proof of the inequalities (\ref{sta}) and (\ref{sta2-}) is analogous
to Lemma 6.1 in \cite{Recke2011}.
\end{proof}
\begin{corollary}
\label{cor6} Under the conditions of lemma \ref{lemma4}, system
(\ref{m1}) -- (\ref{m2}) has $2\tilde{N}$ integral manifolds
\begin{equation}
\mathcal{P}_{j}=\{(\beta t+\vartheta_{j}^{0}+\tilde{v}_{j}(\varphi,\beta t,\alpha t,\mu,\varepsilon),\varphi,t):\ \varphi\in\mathbb{T}_{1},t\in\mathbb{R}\},\label{ma3}
\end{equation}
 where $C^{l-3}$-smooth function
\begin{equation}
\tilde{v}_{j}=\mu^{2}\tilde{v}_{j0}(\beta t,\mu)+\varepsilon^{2}\tilde{v}_{j1}(\varphi,\beta t,\alpha t,\mu,\varepsilon)+\frac{\varepsilon^{3}}{\mu}\tilde{v}_{j2}(\varphi,\beta t,\alpha t,\mu,\varepsilon),
\label{vj}
\end{equation}
 is $2\pi$-periodic in $\varphi,\beta t,$ and $\alpha t$. Manifolds
corresponding to $j=2k,k=1,...,\tilde{N},$ are exponentially stable
for $t\to+\infty$ and manifolds corresponding to $j=2k-1,k=1,...,\tilde{N},$
are exponentially stable for $t\to-\infty.$
\end{corollary}

Note that the expressions for the integral manifolds in the nondegenerate
case, i.e. $g_{1}\not\equiv0$, can be simplified. In particular,
in this case $\tilde{S}_{13}\equiv0$ and $\tilde{S}_{23}\equiv0$
in (\ref{mm1}) -- (\ref{mm2}) as well as $S_{13}\equiv0$ and $S_{23}\equiv0$
in system (\ref{m1}) -- (\ref{m2}), since $R_{04}\equiv0,$ $R_{14}\equiv R_{24}\equiv0$
in system (\ref{ss1}) -- (\ref{ss3}). The following corollary gives
these expressions.

\begin{corollary}
\label{cor66} Let $S_{13}\equiv S_{23}\equiv0$. Assume $\Delta\in[G^{-},G^{+}]$
such that (\ref{alg}) has $2\tilde{N}$ nondegenerate solutions $\xi=\vartheta_{j}^{0}$,
$j=1,\dots,2\tilde{N}$. Then there exist $\mu_{0}>0$ and $c_{0}>0$
such that for all $0<\mu\le\mu_{0}$ and $0\le\varepsilon\le c_{0}$
system (\ref{m1}) -- (\ref{m2}) has $2\tilde{N}$ integral manifolds
\begin{equation}
\mathcal{R}_{j}=\{(\beta t+\vartheta_{j}^{0}+\tilde{v}_{j}(\varphi,\beta t,\alpha t,\mu,\varepsilon),\varphi,t):\ \varphi\in\mathbb{T}_{1},t\in\mathbb{R}\},\label{ma33}
\end{equation}
 where
\[
v_{j}=\mu^{2}v_{j0}(\beta t,\mu)+\varepsilon^{2}v_{j1}(\varphi_{2},\beta t,\alpha t,\mu,\varepsilon),
\]
 with $C^{l-2}$ smooth, periodic in $\varphi_{2},\beta t,\alpha t$
functions $v_{jk}$, such that $\left\Vert v_{jk}\right\Vert _{C^{l-2}}\le M_{3}$
with the constant $M_{3}$ independent on $\alpha,\mu,\varepsilon.$
The manifolds $\mathcal{R}_{2k}$, $k=1,...,\tilde{N},$ are exponentially
stable in the sense of inequality (\ref{sta}). The manifolds $\mathcal{R}_{2k-1}$,
$k=1,...,\tilde{N},$ are exponentially unstable in the sense of inequality
(\ref{sta2-}).
\end{corollary}

\section{\label{sec:Proofs-o-theorems}Proofs of theorems}

We start by proving the degenerate case, since the non-degenerate
case will be a particular case of this proof.

\subsection{Degenerate case $g_{1}=0$}

Let us prove Theorem \ref{theorem011}. In the case $g_{1}=0$, two
averaging transformations (\ref{change1}) and (\ref{change2}) reduce
system (\ref{01}) into (\ref{av010}) with $\gamma^{2}/\alpha=\mu^{2},$
$1/\alpha=\varepsilon^{2}$, and $\bar{g}=g_{2}.$ The obtained system
(\ref{av010}) is further transformed using the local coordinates
to (\ref{ss1}) -- (\ref{ss3}). In the latter system, for any fixed $\beta$ there exists
an orbitally asymptotically stable integral manifold (\ref{mani})
accordingly to Lemma \ref{lemma2}.
Taking into account regular dependence of the right-hand side of
system (\ref{ss1}) -- (\ref{ss3}) on $\beta\in [\beta_1,\beta_2]$ we conclude
that
 constants $\mu_0$, $\varepsilon_0$, $M_1$, $\kappa$, and $L$ from
Lemma \ref{lemma2} can be chosen common  for all $\beta\in [\beta_1,\beta_2]$.
  Hence,
  since dependence (\ref{change2p}),
  in the original system
(\ref{01}) there exists the exponentially stable integral manifold
\begin{eqnarray} \label{mno}
x = e^{A\varphi}(x_0(\psi) + \Phi_0(\psi)u) + \frac{\gamma}{\alpha}\tilde u_2(x_0(\psi) +
\Phi_0(\psi)u, \beta t,\alpha t, \frac{\gamma^2}{\alpha}, \frac{1}{\alpha}),
\end{eqnarray}
where $u$ is defined by (\ref{mani00}). Taking
into account the properties of functions $\tilde u_2$ and $u$, we conclude that
this manifold has form
 (\ref{eq:IM0-deg}).
The equations on the manifold is obtained by substitution (\ref{mani})
into (\ref{ss1}) -- (\ref{ss3}) and are given by (\ref{m1}) --
(\ref{m2}). Substituting the parameters $\mu=\gamma/\sqrt{\alpha}$
and $\varepsilon=1/\sqrt{\alpha}$, one obtains the system on the
manifold (\ref{eq:psi}) -- (\ref{eq:theta}).

All solutions from some neighborhood of the torus $\mathcal{T}_{2}$
are approaching $\mathfrak{M}_{2}(\alpha,\beta,\gamma)$. Therefore,
in order to show the remaining statements of Theorem \ref{theorem011},
it is enough to consider the behavior of solutions on the manifold
$\mathfrak{M}_{2}(\alpha,\beta,\gamma)$. Note that the manifold $\mathfrak{M}_{2}(\alpha,\beta,\gamma)$
corresponds to the manifold $\mathfrak{M}_{\mu,\varepsilon}$ in the
local coordinate system $(h,\psi,\varphi)$ with parameters $\mu$
and $\varepsilon$.

We first consider system (\ref{m1}) -- (\ref{m2}) with parameters
$\mu$ and $\varepsilon$. Let us fix any positive $\epsilon$. For
the set $S_{2}$ of singular values of $G_{2}$ we define the following
sets
\[
\mathcal{B}_{2}(\epsilon)=\{g\in[G_{2}^{-},G_{2}^{+}]:\ \mathrm{dist}(g,S_{2})\ge\epsilon\},
\]
 and
\[
\mathcal{A}_{2}(\epsilon)=\{\theta\in[0,2\pi]:\ G_{2}(\theta)\in\mathcal{B}_{2}(\epsilon)\}.
\]
 Taking into account that the sets $\mathcal{B}_{2}(\epsilon)$ and
$\mathcal{A}_{2}(\epsilon)$ are compact, there exists a positive
constant $\varsigma$ such that
\begin{equation}
\left|\frac{dG_{2}(\theta)}{d\theta}\right|\ge\varsigma\quad\mbox{for all}\quad\theta\in\mathcal{A}_{2}(\epsilon).\label{estpro}
\end{equation}

By Corollary \ref{cor6}, for fixed $\Delta\in\mathcal{B}_{2}(\epsilon)$,
there exist $\mu_{0}>0$ and $c_{0}>0$ such that for all $0<\mu\le\mu_{0}$
and $\varepsilon\le c_{0}\mu^{1/3}$ system (\ref{m1}) -- (\ref{m2})
has $2\tilde{N}$ integral manifolds $\mathcal{P}_{j}$. Since the
estimate (\ref{estpro}) is uniform, constants $\mu_{0}>0$ and $c_{0}>0$
can be chosen common for all $\Delta\in\mathcal{B}(\epsilon),$ and
therefore for all $\alpha,\gamma$, and $\beta$ satisfying  (\ref{cond33}). Hence, under conditions (\ref{eq:cond00}) and
(\ref{cond33}), there exist $2\tilde{N}$ integral
submanifolds $\mathfrak{L}_{j}\subset\mathfrak{M}_{2},$ $j=1,...,2\tilde{N}$
of the form (\ref{xm-deg}) for system (\ref{01}). Locally, the manifolds
corresponding to $j=2k,$ $k=1,\dots,\tilde{N}$ are exponentially
stable and the others are exponentially unstable.

In order to complete the proof, let us show, that any solution from
the manifold $\mathfrak{M}_{2}(\alpha,\beta,\gamma)$ is attracted
to one of the manifolds $\mathfrak{L}_{2k}$ as $t\to\infty$. For
this, it is more convenient to consider the reduced system on the
manifold (\ref{mm1}) -- (\ref{mm2}) and the equivalent to $\mathfrak{L}_{j}$
integral manifolds $\Pi_{j}$ (\ref{ma}).

All manifolds $\Pi_{2k}$, $k=1,\dots,\tilde{N}$, are asymptotically
stable with the exponential estimate (\ref{sta}) and manifolds $\Pi_{2k-1}$
are asymptotically unstable with the estimate (\ref{sta2-}). Therefore,
by (\ref{sta2-}), on a finite time interval $T$, the solution $(\psi_{2}(t),\varphi_{2}(t))$
of (\ref{mm1}) -- (\ref{mm2}) with initial values $(\psi_{2}(t_{0}),\varphi_{2}(t_{0}))$
not belonging to the manifold $\Pi_{2k-1}$, i.e.
\[
\psi_{2}(t_{0})\neq\vartheta_{2k-1}^{0}+v_{2k-1}(\varphi_{2}(t_{0}),\beta t_{0},\alpha t_{0},\mu,\varepsilon),
\]
 and $|\psi_{2}(t_{0})-\vartheta_{2k-1}^{0}|<\delta_{0}$, reaches
the boundary of $\delta_{0}$-neighborhood of unperturbed manifold
$\vartheta_{2k-1}^{0}\times\mathbb{T}_{1}$, more exactly, the value
of $\psi_{2}(t)$ reaches $\vartheta_{2k-1}^{0}-\delta_{0}$ or $\vartheta_{2k-1}^{0}+\delta_{0},$
where $\delta_{0}$ is defined from Lemma \ref{lemma4}. The time
interval $T$ depends on the initial values $(\psi_{2}(t_{0}),\varphi_{2}(t_{0}))$
and parameters $\mu,\varepsilon$.

Further, any solution $(\psi_{2}(t),\varphi_{2}(t))$ starting outside
of a $\delta_{0}$-neighborhood of $\vartheta_{2k-1}^{0}\times\mathbb{T}_{1}$
reaches $\delta_{0}$-neighborhood of either the manifold $\vartheta_{2k}^{0}\times\mathbb{T}_{1}$
or $\vartheta_{2k-2}^{0}\times\mathbb{T}_{1}$ on a finite time interval.
This follows from the fact, that the right hand side of (\ref{mm1})
is uniformly bounded from zero on this set for small enough parameters
$\varepsilon$ and $\mu$ (see also Lemma 6.2, \cite{Recke2011}).

Next, by the inequality (\ref{sta}) of Lemma \ref{lemma4}, any solution
from the $\delta_{0}$-neighborhood of the manifold $\vartheta_{2k}^{0}\times\mathbb{T}_{1}$
is attracted to the stable integral manifold $\Pi_{2k}.$

As a result, solutions $(\psi(t),\varphi(t))$ of system (\ref{m1})
-- (\ref{m2}) that, at initial point $t=t_{0}$ do not belong to
unstable integral manifolds $\mathcal{P}_{2k-1},$ $k=1,...,N,$ are
attracted for $t\ge t_{0}$ to solutions $(\bar{\psi}(t),\bar{\varphi}(t))$
on one of the stable integral manifolds $\mathcal{P}_{2k}$, i.e.
\[
\bar{\psi}(t)=\beta t+\vartheta_{2k}^{0}+\tilde{v}_{2k}(\bar{\varphi}(t),\beta t,\alpha t,\mu,\varepsilon),
\]
 and $\bar{\varphi}(t)$ satisfies (\ref{m2}) with $\psi(t)=\bar{\psi}(t)$
and $j=2k$
\begin{multline}
\frac{d\varphi}{dt}=\alpha_{0}+\mu^{2}S_{21}(\beta t+\vartheta_{j}^{0}+\tilde{v}_{j},\beta t,\mu)+\varepsilon^{2}\mu^{2}S_{22}(\beta t+\vartheta_{j}^{0}+\tilde{v}_{j},\varphi,\beta t,\alpha t,\mu,\varepsilon)\\
+\varepsilon^{3}\mu S_{23}(\beta t+\vartheta_{j}^{0}+\tilde{v}_{j},\varphi,\beta t,\alpha t,\mu,\varepsilon).\label{ma3na}
\end{multline}

If a solution $(\psi(t),\varphi(t))$ of (\ref{m1}) -- (\ref{m2})
at initial point $t=t_{0}$ belongs to one of the unstable integral
manifolds $\mathcal{P}_{2k+1}$ then this solution has the following
form
\[
\psi(t)=\beta t+\vartheta_{2k+1}^{0}+\tilde{v}_{2k+1}(\varphi(t),\beta t,\alpha t,\mu,\varepsilon),
\]
 where $\varphi(t)$ satisfies (\ref{ma3na}) for $j=2k+1$.

Lemma \ref{lemma2} implies, that any solution $(h(t),\psi(t),\varphi(t))$
of (\ref{ss1}) -- (\ref{ss3}) with $\|h(t_{0})\|\le\nu_{0}$ is
attracted to one of the solutions $(\bar{h}(t),\bar{\psi}(t),\bar{\varphi}(t))$
on the integral manifold $\mathfrak{M}_{\mu,\varepsilon}$ such that
\begin{eqnarray*}
 &  & \bar{h}(t)=u(\bar{\psi}(t),\bar{\varphi}(t),\beta t,\alpha t,\mu,\varepsilon),\\[1mm]
 &  & \bar{\psi}(t)=\beta t+\vartheta_{j}^{0}+\tilde{v}_{j}(\bar{\varphi}(t),\beta t,\alpha t,\mu,\varepsilon),\\[1mm]
 &  & \bar{\varphi}(t)\mbox{\,\,\ is a solution of system (\ref{ma3na})}
\end{eqnarray*}
 with some $j,\ 1\le j\le2\tilde{N}.$
Therefore, every solution $x_2(t)$
of averaged system (\ref{av2}) that at a certain moment of time $t_{0}$ belong
to a small neighborhood of invariant torus  $\mathcal{T}_{2}$ tends to one
of the solutions
\[
\bar{x}_2(t)=e^{A\bar{\varphi}(t)}\left(x_{0}(\beta t+\vartheta_{j}^{0}+\tilde{v}_{j})+\Phi_{0}(\beta t+\vartheta_{j}^{0}+\tilde{v}_{j})\bar{h}(t)\right),\quad j=1,\dots,2\tilde{N}.
\]
Taking into account averaging transformation (\ref{change2p}),
and the form (\ref{vj}) of $\bar{v}_j$,
every solution $x(t)$ of system (\ref{01}) that at a certain moment of time
$t_{0}$ belongs
to $\delta$-neighborhood of torus $\mathcal{T}_{2}$ tends  as $t\to+\infty$
to one
of the manifolds $\mathfrak{L}_{j}$ given by (\ref{xm-deg}).

The proof of theorem \ref{thm:main2-1}
follows from (\ref{xm-deg}) taking into
account smallness of $\gamma/\alpha$, $\gamma^{2}/\alpha$, and $\frac{1}{\gamma\alpha}$ for
$\alpha,\beta$, and $\gamma$ satisfying (\ref{cond33}).

\subsection{Non-degenerate case}

If $g_{1}\not\equiv0,$ the averaging transformation (\ref{change1})
in (\ref{01}) leads to system (\ref{av010}) with $\gamma=\mu^{2},$
$1/\alpha=\varepsilon^{2}$, and $\bar{g}=g_{1}$, $r_{2}\equiv0$.
The existence of the asymptotically stable invariant manifold $\mathfrak{M}_{1}(\alpha,\beta,\gamma)$
of the form (\ref{eq:IM0}) follows directly from Lemma \ref{lemma2}
and the coordinate transformation (\ref{eq:zamina}).

For the proof of the other statements of Theorem \ref{theorem01}
we consider the dynamics on the manifold $\mathfrak{M}_{1}(\alpha,\beta,\gamma).$
For any $\alpha,\gamma$ and $\beta$ satisfying 
(\ref{cond2}) there exists an even number of points $\vartheta_{j}^{0},$
$j=1,...,2\tilde{N},$ which are solutions of the equation $\beta-\beta_{0}=\gamma G_{1}(\theta)$.
The number $\tilde{N}$ depends on the parameters $\alpha,\gamma$
and $\beta$.

By Corollary \ref{cor66}, for all $\mu\in(0,\mu_{0}]$ and $\varepsilon\le c_{0}$
with sufficiently small $\mu_{0}$ and $c_{0}$ system (\ref{m1})
-- (\ref{m2}) has $2\tilde{N}$ integral manifolds $\mathcal{R}_{j}$
corresponding to points $\vartheta_{j}^{0}$. Analogously to the proof
of Theorem \ref{theorem011} in the degenerate case, we show that
every solution $(\psi(t),\varphi(t))$ of system (\ref{m1}) -- (\ref{m2})
is attracted for $t\to+\infty$ to some solution $(\bar{\psi}(t),\bar{\varphi}(t))$
on one of the integral manifolds $\mathcal{R}_{j},$ $j=1,...,2\tilde{N}.$

The manifolds $\mathcal{R}_{j}$ in system (\ref{m1}) -- (\ref{m2})
correspond to the manifolds $\mathfrak{N}_{j}$ in system (\ref{01}).
Therefore, every solution $x(t)$ of system (\ref{01}) that at a
certain moment of time $t_{0}$ belong to $\delta$-neighborhood of
torus $\mathcal{T}_{2}$ tends to one of the manifolds $\mathfrak{N}_{j}$
as $t\to+\infty$ and satisfies estimate (\ref{theo3}) from Theorem
\ref{thm:main}.

\section{Example 1\label{sec:example1}}

In this section, we illustrate the obtained results using a model
of a laser with saturable absorber with an external electro-optical
forcing. The model has the form (\ref{01}) with $x\in\mathbb{R}^{4}$,
\begin{equation}
f(x)=\left(\begin{array}{c}
\mu\left(a-x_{1}-x_{1}\left(x_{3}^{2}+x_{4}^{2}\right)\right)\\
\mu\left(b-x_{2}-cx_{2}\left(x_{3}^{2}+x_{4}^{2}\right)\right)\\
\frac{1}{2}\left(x_{1}-x_{2}-1\right)x_{3}-\alpha_{0}x_{4}\\
\frac{1}{2}\left(x_{1}-x_{2}-1\right)x_{4}+\alpha_{0}x_{3}
\end{array}\right),\label{eq:f}
\end{equation}
where $\mu,$ $a$, $b$, $c$, and $\alpha_{0}$ are real  parameters.
This system satisfies the symmetry condition (\ref{eq:symf}) with matrix
\begin{equation}
A=\left(\begin{array}{cccc}
0 & 0 & 0 & 0\\
0 & 0 & 0 & 0\\
0 & 0 & 0 & -1\\
0 & 0 & 1 & 0
\end{array}\right).\label{eq:A}
\end{equation}
 Note, that in the unperturbed case, the model can be written in partially complex
form
\begin{eqnarray}
x_{1}' & = & \mu\left(a-x_{1}-x_{1}\left|y\right|^{2}\right),\label{eq:yam1}\\
x_{2}' & = & \mu\left(b-x_{2}-cx_{2}\left|y\right|^{2}\right),\label{eq:yam2}\\
y' & = & \frac{1}{2}\left(x_{1}-x_{2}-1\right)y+i\alpha_{0}y,\label{eq:yam3}
\end{eqnarray}
where $y=x_{3}+ix_{4}$, $i=\sqrt{-1}$, which can be further reduced
to the Yamada model describing the dynamics of the intensity of a
laser with saturable absorber \cite{Yamada1993,Dubbeldam1999}. In
this case, $y$ plays the role of the electric field with intensity
$I=|y|^{2}=x_3^2+x_4^2$, and $x_{1}$, $x_{2}$ are gain and absorption, respectively.
In \cite{Dubbeldam1999}, the parameter regions are numerically obtained,
for which the
Yamada model (in coordinates $(x_1,x_2,I)$) has an asymptotically stable periodic solution
$\left(x_{1}^{0}(\beta_0 t),x_{2}^{0}(\beta_0 t),I^{0}(\beta_0  t)\right)^{T}$, $I^{0}(\beta_0 t)>0$.
We choose the parameter values $a=7.0$, $b=5.8$, $c=1.8$, and $\mu=0.04$
belonging to this region. For these parameters, $\beta_0=0.071$ and
the corresponding complexified system (\ref{eq:yam1})
-- (\ref{eq:yam3}) has the quasiperiodic solution
$x(t)=\left(x_{1}^{0}(\beta_{0}t),x_{2}^{0}(\beta_{0}t),\sqrt{I^{0}(\beta_{0}t)}
e^{i\left(\alpha_{0}t+\xi\right)}\right)^{T}$,
where  $\xi$ is an arbitrary real constant resulting from the symmetry.
Note that the $S^{1}$- equivariance of the right-hand side of (\ref{eq:yam1})
-- (\ref{eq:yam3}) with respect to the transformation $y\to e^{i\xi}y$
is equivalent to the $S^{1}$-equivariance (\ref{eq:symf}) of the function (\ref{eq:f})
with matrix $A$ from (\ref{eq:A}). The corresponding
solution (\ref{qp}) of (\ref{02}) with $f$ defined by (\ref{eq:f}) reads then
\[
x(t)=\left(x_{1}^{0}(\beta_{0}t),x_{2}^{0}(\beta_{0}t),\sqrt{I^{0}(\beta_{0}t)}\cos\left(\alpha_{0}t+\xi\right),\sqrt{I^{0}(\beta_{0}t)}\sin\left(\alpha_{0}t+\xi\right)\right)^{T}.
\]

In this way, using numerical results from \cite{Dubbeldam1999}, we see
that the unperturbed system (\ref{02}) with (\ref{eq:f}) has an exponentially
attracting invariant torus (\ref{eq:torus}) with $A$ from (\ref{eq:A}) and
\begin{equation}
\label{x0yam}
x_{0}(\psi)=\left(x_{1}^{0}(\psi),x_{2}^{0}(\psi),x_{3}^{0}(\psi)=\sqrt{I^{0}(\psi)},0\right)^{T}.
\end{equation}
Here $x_{0}(\psi)$ is the periodic solution of the system
(\ref{newcoord}) transformed to the corotating coordinates.
It can be found numerically using direct integration. Figure
\ref{fig:x0} illustrates $x_{0}(\psi)$ for the chosen parameter
values. It has shape of a pulse typical for lasers with saturable
absorber.

\begin{figure}
\begin{centering}
\bigskip
\smallskip
\includegraphics[angle=-90,width=0.8\linewidth]{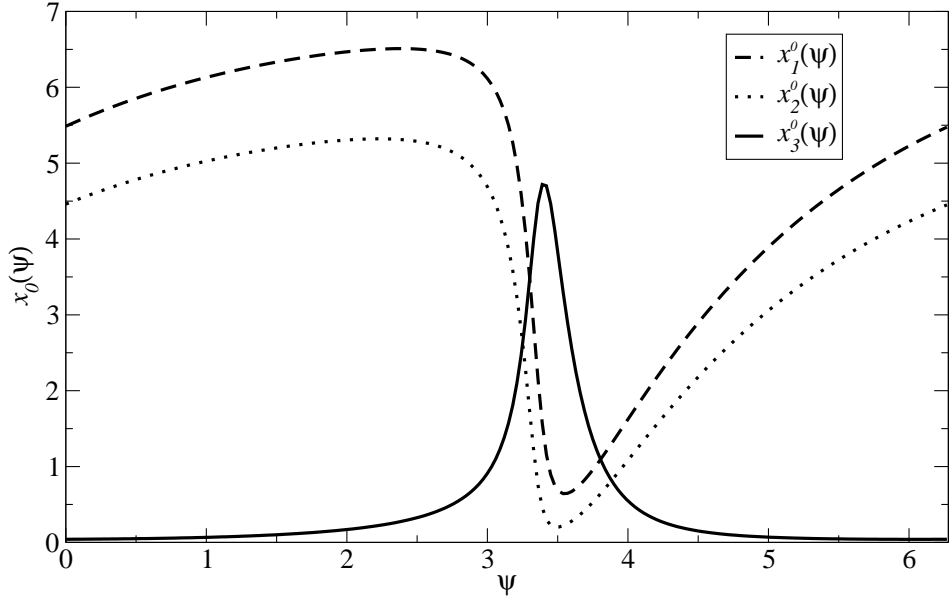}
\par\end{centering}
\caption{\label{fig:x0} Exponentially stable periodic solution $x_{0}(\psi)$
of the unperturbed system (\ref{newcoord}) with the right hand side (\ref{eq:f}).
The components $x_{1}^{0}$, $x_{2}^{0}$, and $x_{3}^{0}$ are shown,
while $x_{4}^{0}(\psi)\equiv0$. }
\end{figure}

Accordingly to Theorems \ref{thm:main} and \ref{thm:main2-1},
the synchronization region for
the parameters $\alpha,\beta,$ and $\gamma$ is determined by the
properties of function $G_{1}(\psi)$ from (\ref{0per3}) in the nondegenerate
case and $G_{2}(\psi)$ from (\ref{01per3}) in the degenerate case. In both
cases, one needs the periodic solution $p_{1}\left(\psi\right)$ of
the adjoint variational equation (\ref{0per2}),
which satisfies the orthonormality
condition (\ref{24new}). Using standard numerical methods,
which involve computation of the monodromy matrix, we found $p_{1}\left(\psi\right)$
numerically. Because of (\ref{24new}), $p_{1}$ is
defined uniquely. The result is shown in Figure \ref{fig:P1}.

\begin{figure}

\begin{centering}
\bigskip
\smallskip
\includegraphics[angle=-90,width=0.8\linewidth]{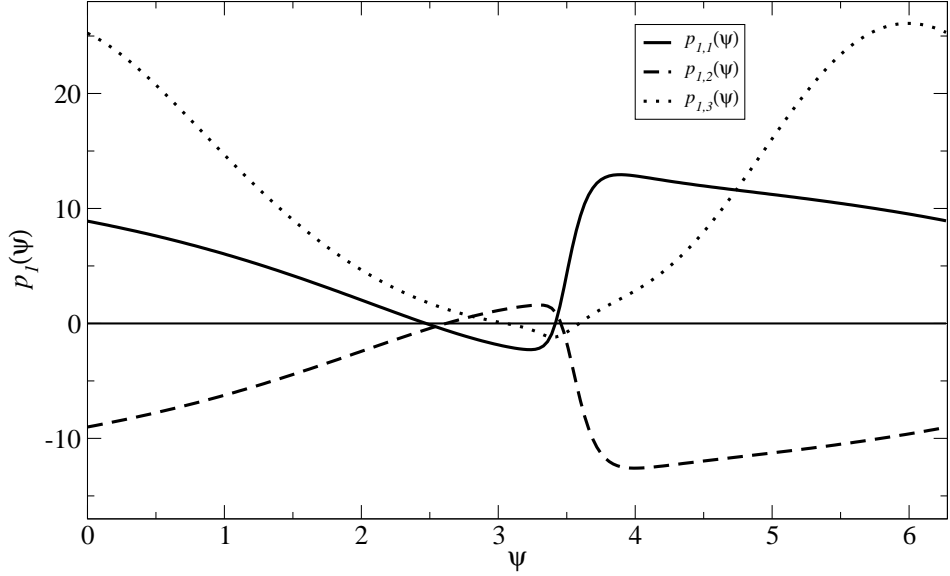}
\par\end{centering}

\caption{\label{fig:P1}Periodic solution $p_{1}\left(\psi\right)$ of the
adjoint variational system (\ref{0per2}) for unperturbed Yamada model. The
components $p_{1,1}$, $p_{1,2}$, and $p_{1,3}$ are shown, while
$p_{1,4}(\psi)\equiv0$. }

\end{figure}

With the given $p_{1}\left(\psi\right)$, the effects of arbitrary
perturbation of the form (\ref{01}) -- (\ref{eq:symg}) can be studied.
For example, let us consider the perturbation
\begin{equation}
\gamma g\left(x,\beta t,\alpha t\right)=\gamma\left(g_{\mathrm{el}}\left(\beta t\right),0,\cos\left(\alpha t\right)g_{\mathrm{op}}\left(\beta t\right),\sin\left(\alpha t\right)g_{\mathrm{op}}\left(\beta t\right)\right)^{T}.\label{eq:pert}
\end{equation}
This kind of perturbation may correspond to some electro-optical external
injection, where $\gamma g_{\mathrm{el}}\left(\alpha t\right)$ stands
for an electric and $\gamma\cos\left(\alpha t\right)g_{\mathrm{op}}\left(\beta t\right)$,
$\gamma\sin\left(\alpha t\right)g_{\mathrm{op}}\left(\beta t\right)$
for optical components, respectively. Here, we consider (\ref{eq:pert})
just as an illustrative example for our method.

The function $g_{1}(x,\psi)$ from (\ref{av}) is then reduced to
\[
g_{1}\left(x,\psi\right)=[g_{\mathrm{el}}(\psi),0,0,0]^{T}
\]
and
\[
G_{1}\left(\psi\right)=\frac{1}{2\pi}\int_{0}^{2\pi}g_{\mathrm{el}}(\theta - \psi)p_{1,1}\left(\theta\right)d\theta.
\]
Accordingly to Theorems \ref{theorem01} and \ref{thm:main},
the locking regions have the
form as in Fig. \ref{fig:generic}
and the corresponding angles are determined by
the extrema of $G_{1}\left(\psi\right)$. Let us consider, for example,
$g_{\mathrm{el}}\left(\psi\right)=\sin\left(\psi\right)$. Then the
function $G_{1}(\psi)$ has the form
\[
G_{1}\left(\psi\right)=G_{s}\sin\psi+G_{c}\cos\psi
\]
with
\[
G_{c}=-\frac{1}{2\pi}\int_{0}^{2\pi}p_{1,1}\left(\theta\right)\cos\theta d\theta,\quad G_{s}=\frac{1}{2\pi}\int_{0}^{2\pi}p_{1,1}\left(\theta\right)\sin\theta d\theta.
\]
For the chosen parameter values (using $p_{1}(\psi)$ given in Fig.~\ref{fig:P1}),
we have $G_{c}\approx1.67$ and $G_{s}\approx-2.33.$ Hence, $G_{1}\left(\psi\right)$
is a harmonic function with extrema $G_{1}^{+}=\sqrt{G_{s}^{2}+G_{c}^{2}}\approx2.86$
and $G_{1}^{-}=-G_{1}^{+}$, see Fig.~\ref{fig:G1}(a). The synchronization
domain is symmetric and given by the conditions $\alpha>\alpha_{1},$
$0<\gamma<\gamma_{1}$, and $|\beta-\beta_{0}|<\gamma G_{1}^{+}\approx2.86\gamma$.
The bounds $\gamma_{1}$ and $\alpha_{1}$ cannot be determined using our approach.
If parameters belongs to synchronization domain, then for any solution $x(t)$ of system (\ref{01}) with right-hand side defined by
(\ref{eq:f}) and (\ref{eq:pert}) which is at a certain moment close to invariant torus
(\ref{eq:torus}) of unperturbed system with $x_0$ from (\ref{x0yam}),
there exists $\sigma \in \mathbb{R}$
such that
$$|x_1(t) - x_1^0(\beta t + \sigma)| + |x_2(t) - x_2^0(\beta t + \sigma)| +
|x_3^2(t) + x_4^2(t) -I^0(\beta t + \sigma)| \approx 0$$
for large $t.$

Interestingly, when the electric perturbation has two
maxima, e.g. $g_{\mathrm{el}}(\psi)=0.5\cos\psi+\sin(2\psi)$, then
$G_{1}\left(\psi\right)$ can also have two maxima (see Fig.~\ref{fig:G1}(b))
which leads to the appearance of multiple stable synchronized manifolds
(Theorem \ref{theorem01}).
In this case, the set of singular values is
$S_{1}=\left\{ G_{1}^{+}=1.87,1.22,0.01,-2.77=G_{1}^{-}\right\} $
and the synchronization domain is given by (\ref{cond2}).

\begin{figure}
\begin{centering}
\bigskip
\smallskip
\includegraphics[angle=-90,width=0.8\linewidth]{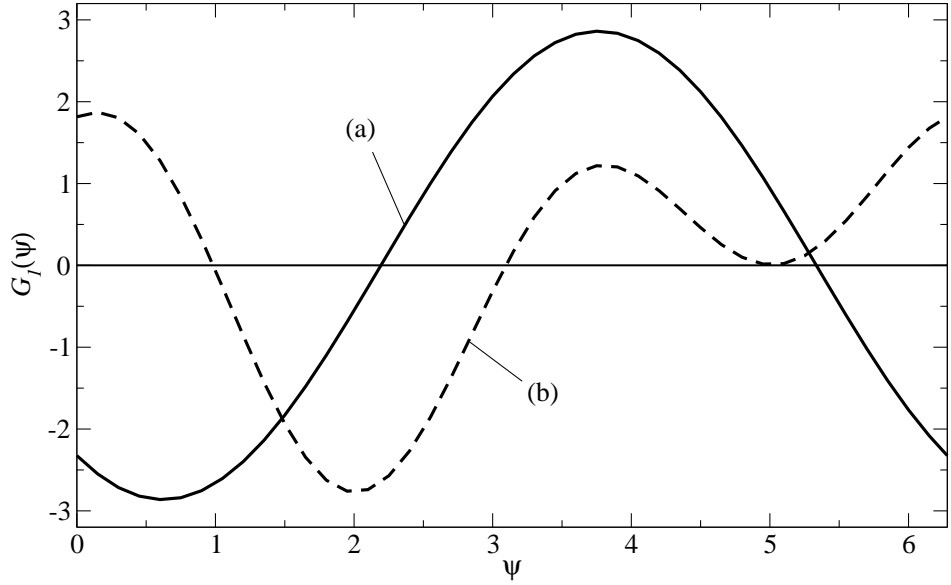}
\par\end{centering}

\caption{\label{fig:G1}Function $G_{1}\left(\psi\right)$ determining the
locking regions for the Yamada model with perturbation (\ref{eq:pert}).
(a) $g_{\mathrm{el}}=\sin\left(\psi\right)$, (b) $g_{\mathrm{el}}=0.5\cos\psi+\sin\left(2\psi\right).$}
\end{figure}

In a similar way, the influence of other perturbations of the form
(\ref{01}) -- (\ref{eq:symg}) can be studied.

\section{Example 2\label{sec:example2}}
In the next example, the system has quasiperiodic solution of the form
(\ref{qp}),
which can be written in simple analytical form. We consider system
(\ref{01}) with $x = (x_1,x_2,x_3,x_4) \in \mathbb{R}^4$ and
\begin{eqnarray*}
& & f(x) = \left(
           \begin{array}{c}
             \beta_0 x_2 + x_1(x_3^2 + x_4^2 - x_1^2 - x_2^2) \\
             -\beta_0 x_1 + x_2(x_3^2 + x_4^2 - x_1^2 - x_2^2) \\
(1 - x_1^2 - x_2^2)x_3 + ( \beta_0 x_1 +  \beta_0 x_2 + \alpha_0)x_4 \\
-( \beta_0 x_1 +  \beta_0 x_2 + \alpha_0)x_3 + (1 - x_1^2 - x_2^2)x_4 \\
\end{array} \right), \\
& & g(x,\psi, \varphi) =
\left( \begin{array}{c} x_3\cos\varphi + x_4\sin\varphi \\ -x_3\sin\varphi + x_4\cos\varphi \\
 (x_3^2+x_4^2)\sin\psi \cos\varphi \\
(x_3^2+x_4^2)\sin\psi \sin\varphi \\
\end{array} \right).
\end{eqnarray*}
This system satisfies symmetry conditions (\ref{eq:symf}) and (\ref{eq:symg}) with matrix $A$ defined by (\ref{eq:A}).
Unperturbed system (if $\gamma = 0$) has the solution
 \begin{eqnarray*} x(t) =  e^{A\alpha_0 t} x_0(\beta_0 t) = e^{A\alpha_0 t}
\left( \begin{array}{c} \sin(\beta_0 t) \\
\cos(\beta_0 t) \\
\sin\left(\sin(\beta_0 t) - \cos(\beta_0 t)\right) \\
\cos\left(\sin(\beta_0 t) - \cos(\beta_0 t)\right) \\
\end{array} \right)
= e^{A\alpha_0 t}
\left( \begin{array}{c} x_{1}^0(\beta_0 t) \\
x_{2}^0(\beta_0 t) \\
x_{3}^0(\beta_0 t) \\
x_{4}^0(\beta_0 t) \\
\end{array} \right).
\end{eqnarray*}
The corresponding variational equation (\ref{0per1})
 has two linear independent periodic solution
$$q_1(\psi) = \frac{dx_0(\psi)}{d\psi}, \quad q_2(\psi) = A x_0(\psi).$$
The adjoint equation
(\ref{0per2}) has two periodic solution $p_1(\psi)$ and $p_2(\psi)$ such that $p_j^T(\psi)q_k(\psi) = \delta_{jk}$ for all $\psi$ and $j,k = 1,2.$ It can be verified that $p_1(\psi) = (\cos\psi, -\sin\psi, 0, 0).$ By direct computation we find that $g_1\equiv 0$ and the second averaging function
$$
g_2(x_0(\psi),\psi) =
(0, -\sin\psi,
-x_{4}^0(\psi)\sin^2\psi, x_{3}^0(\psi)\sin^2\psi)
$$ and
\begin{eqnarray*} & & G_2(\psi) = \frac{1}{2\pi}\int_0^{2\pi}p_1^T(\psi + \theta)g_2(x_0(\psi + \theta),\theta)d\theta = \\ & & = \frac{1}{2\pi}\int_0^{2\pi} \sin(\psi + \theta)\sin\theta d\theta =  \frac{1}{2} \cos\psi. \end{eqnarray*}
 Therefore, $G_2^+ = -G_2^- = 1/2$ and by Theorem 2.3 the synchronization domain (where condition (\ref{theo6}) is fulfilled) is defined by $|\beta - \beta_0| < \gamma^2 G_2^+/\alpha = \gamma^2 / 2\alpha$ for $\alpha > \alpha_1$ and $c_1/ \alpha \le \gamma \le c_2 \sqrt{\alpha}$ with some
positive constants $\alpha_1$, $c_1$, and $c_2$.

 \bibliographystyle{siam}

\begin{thebibliography}{10}

\bibitem{Bandelow1998}
{\sc U.~Bandelow, L.~Recke, and B.~Sandstede}, {\em Frequency regions for
  forced locking of self-pulsating multi-section {DFB} lasers}, Opt. Commun.,
  147 (1998), pp.~212--218.

\bibitem{Bogoliubov1961}
{\sc N.~N Bogoliubov and Yu.~A. Mitropolskii}, {\em Asymptotic Method in the
  Theory of Nonlinear Oscillations}, Gordon and Breach, New York, 1961.

\bibitem{Chicone2006}
{\sc C.~Chicone}, {\em Ordinary Differential Equations with Applications, 2nd
  Edition}, Texts in Applied Mathematics, Springer, 2006.

\bibitem{Chillingworth2000}
{\sc D.~Chillingworth}, {\em Generic multiparameter bifurcation from a
  manifold}, Dyn. Stab. Syst., 15 (2000), pp.~101--137.

\bibitem{Crawford1988}
{\sc J.~D. Crawford, M.~Golubitsky, and W.~F. Langford}, {\em Modulated
  rotating waves in $\mathcal{O}(2)$ mode interactions}, Dyn. Stab. Syst., 3
  (1988), pp.~159--175.

\bibitem{Demidovich1967}
{\sc B.~P. Demidovich}, {\em Lectures on Stability Theory}, Nauka, Moscow,
  1967.
\newblock in Russian.

\bibitem{Dubbeldam1999}
{\sc J.~L.~A. Dubbeldam and B.~Krauskopf}, {\em Self-pulsations of lasers with
  saturable absorber: dynamics and bifurcations}, Opt. Commun., 159 (1999),
  pp.~325--338.

\bibitem{Feiste1994}
{\sc U.~Feiste, D.~J.~As, and A.~Erhardt}, {\em 18 ghz all-optical frequency
  locking and clock recovery using a self-pulsating two-section laser}, IEEE
  Photon. Technol. Lett., 6 (1994), pp.~106--108.

\bibitem{Field2007}
{\sc M.~J. Field}, {\em Dynamics and Symmetry}, Imperial College Press, 2007.

\bibitem{Hale1980}
{\sc J.~K. Hale}, {\em Ordinary differential equations}, Second edition. 
Krieger Publ. Co., 1980. 

\bibitem{Husemoller1993}
{\sc D.~Husemoller}, {\em Fibre Bundles}, Springer, 1993.

\bibitem{Lichtner2007}
{\sc M.~Lichtner, M.~Radziunas, and L.~Recke}, {\em Well-posedness, smooth
  dependence and center manifold reduction for a semilinear hyperbolic system
  from laser dynamics}, Math. Methods Appl. Sci., 30 (2007), pp.~931--960.

\bibitem{Nizette2001}
{\sc M.~Nizette, T.~Erneux, A.~Gavrielides, and V.~Kovanis}, {\em Stability and
  bifurcations of periodically modulated, optically injected laser diodes},
  Phys. Rev. E, 63 (2001), p.~026212.

\bibitem{Peterhof1999}
{\sc D~Peterhof and B~Sandstede}, {\em All-optical clock recovery using
  multisection distributed-feedback lasers}, J. Nonlinear Sci., 9 (1999),
  pp.~575--613.

\bibitem{Radziunas2006}
{\sc M.~Radziunas}, {\em Numerical bifurcation analysis of the traveling wave
  model of multisection semiconductor lasers}, Physica D, 213 (2006),
  pp.~98--112.

\bibitem{Rand1982}
{\sc D.~Rand}, {\em Dynamics and symmetry. {P}redictions for modulated waves in
  rotating fluids}, Arch. Rat. Mech. Anal., 79 (1982), pp.~1--37.

\bibitem{Recke1998a}
{\sc L.~Recke}, {\em Forced frequency locking of rotating waves}, Ukrain. Math.
  J, 50 (1998), pp.~94--101.

\bibitem{Recke1998}
{\sc L.~Recke and D.~Peterhof}, {\em Abstract forced symmetry breaking and
  forced frequency locking of modulated waves}, J. Differential Equations, 144
  (1998), pp.~233--262.

\bibitem{Recke2011}
{\sc L.~Recke, A.~Samoilenko, A.~Teplinsky, V.~Tkachenko, and S.~Yanchuk}, {\em
  Frequency locking of modulated waves}, Discrete and continuous dynamical
  systems, 31 (2011), pp.~847--875.

\bibitem{Samoilenko2005}
{\sc A.~M. Samoilenko and L.~Recke}, {\em Conditions for synchronization of one
  oscillation system}, Ukrain. Math. J., 57 (2005), pp.~1089--1119.

\bibitem{Sanders2007}
{\sc J.~A. Sanders, F.~Verhulst, and J.~Murdock}, {\em Averaging methods in nonlinear dynamical systems}. Second
  edition, Springer, 2007.

\bibitem{Sartorius1998}
{\sc B.~Sartorius, C.~Bornholdt, O.~Brox, H.J. Ehrke, D.~Hoffmann, R.~Ludwig,
  and M.~M{\"o}hrle}, {\em All-optical clock recovery module based on
  self-pulsating dfb laser}, Electronics Letters, 34 (1998), pp.~1664--1665.

\bibitem{Schneider2005}
{\sc K.~R. Schneider}, {\em Entrainment of modulation frequency: a case study},
  Int. J. Bifurc. Chaos Appl. Sci. Eng., 15 (2005), pp.~3579--3588.

\bibitem{Sieber2002}
{\sc J.~Sieber}, {\em Numerical bifurcation analysis for multisection
  semiconductor lasers}, SIAM J. Appl. Dyn. Syst., 1 (2002), pp.~248--270.

\bibitem{Vanderbauwhede1989}
{\sc A.~Vanderbauwhede}, {\em Centre manifolds, normal forms and elementary
  bifurcations}, in Dynamics Reported, Volume 2, U.~Kirchgraber and H.O.
  Walther, eds., John Wiley \& Sons Ltd and B.G. Teubner, 1989, pp.~89--169.

\bibitem{Wieczorek2005}
{\sc S.~Wieczorek, B.~Krauskopf, T.~B. Simpson, and D.~Lenstra}, {\em The
  dynamical complexity of optically injected semiconductor lasers}, Phys. Rep.,
  416 (2005), pp.~1--128.

\bibitem{Yamada1993}
{\sc M.~Yamada}, {\em A theoretical analysis of self-sustained pulsation
  phenomena in narrow stripe semiconductor lasers}, IEEE J. Quantum Electron.,
  QE-29 (1993), p.~1330.

\bibitem{Yi1993a}
{\sc Y.~Yi}, {\em Stability of integral manifold and orbital attraction of
  quasi-periodic motion}, J. Differential Equation, 103 (1993), pp.~278--322.

\end{thebibliography}

 \end{document}